\documentclass[12pt,oneside,english]{amsart}
\usepackage[T1]{fontenc}
\usepackage[latin9]{inputenc}
\usepackage{geometry}
\usepackage{enumitem}
\usepackage{mathrsfs}
\usepackage{amstext}
\usepackage{amsthm}
\usepackage{amssymb}
\usepackage{bm}

\usepackage[svgnames]{xcolor}
\usepackage[bookmarksnumbered=true]{hyperref} 
\hypersetup{
	colorlinks = true,
	linkcolor = Blue,
	anchorcolor = blue,
	citecolor = Green,
	filecolor = blue,
	urlcolor = FireBrick
}
\usepackage{bbm}
\usepackage{comment}

\usepackage{algorithm}
\usepackage{algpseudocode}

\usepackage{aliascnt}
\usepackage[nameinlink]{cleveref}

\makeatletter
\numberwithin{equation}{section}
\numberwithin{figure}{section}

\theoremstyle{plain}
\newtheorem{theorem}{Theorem}[section]

\theoremstyle{definition}
\newaliascnt{definition}{theorem}
\newtheorem{definition}[definition]{Definition}
\aliascntresetthe{definition}
\crefname{definition}{Definition}{Definitions}

\theoremstyle{definition}
\newaliascnt{question}{theorem}

\aliascntresetthe{question}
\crefname{question}{Question}{Questions}

\theoremstyle{definition}
\newaliascnt{remark}{theorem}
\newtheorem{remark}[remark]{Remark}
\aliascntresetthe{remark}
\crefname{remark}{Remark}{Remarks}

\theoremstyle{plain}
\newaliascnt{corollary}{theorem}
\newtheorem{corollary}[corollary]{Corollary}
\aliascntresetthe{corollary}
\crefname{corollary}{Corollary}{Corollaries}

\theoremstyle{plain}
\newaliascnt{lemma}{theorem}
\newtheorem{lemma}[lemma]{Lemma}
\aliascntresetthe{lemma}
\crefname{lemma}{Lemma}{Lemmas}

\theoremstyle{plain}
\newaliascnt{proposition}{theorem}
\newtheorem{proposition}[proposition]{Proposition}
\aliascntresetthe{proposition}
\crefname{proposition}{Proposition}{Propositions}

\theoremstyle{definition}
\newaliascnt{example}{theorem}
\newtheorem{example}[example]{Example}
\aliascntresetthe{example}
\crefname{example}{Example}{Examples}

\theoremstyle{definition}
\newaliascnt{assumption}{theorem}
\newtheorem{assumption}[assumption]{Assumption}
\aliascntresetthe{assumption}
\crefname{assumption}{Assumption}{Assumptions}

\theoremstyle{definition}
\newaliascnt{problem}{theorem}

\aliascntresetthe{problem}
\crefname{problem}{Problem}{Problems}

\newtheorem*{definition*}{Definition}
\newtheorem*{proposition*}{Proposition} 
\newtheorem*{remark*}{Remark}
\newtheorem*{example*}{Example}
\newtheorem*{problem*}{Problem}
\newtheorem*{question*}{Question}
\newtheorem*{theorem*}{Theorem}
\newtheorem*{lemma*}{Lemma}
\newtheorem*{corollary*}{Corollary}
\newtheorem*{Acknowledgments*}{Acknowledgments}

\renewenvironment{proof}[1][\proofname]{\medskip \noindent {\bfseries #1. }}{\hfill \qedsymbol\medskip}

\MakeRobust{\ref}
\newcommand{\labeltext}[2]{
\@bsphack
\csname phantomsection\endcsname
\def\@currentlabel{#1}{\label{#2}}
\@esphack
}

\usepackage{scalerel}
\usepackage[usestackEOL]{stackengine}
\def\dashint{\,\ThisStyle{\ensurestackMath{%
  \stackinset{c}{.2\LMpt}{c}{.5\LMpt}{\SavedStyle-}{\SavedStyle\phantom{\int}}}%
  \setbox0=\hbox{$\SavedStyle\int\,$}\kern-\wd0}\int}

\usepackage{enumitem}

\DeclareRobustCommand{\SkipTocEntry}[5]{}

\newcommand{\mR}{\mathbb{R}}   % Formatting for R
\newcommand{\mN}{\mathbb{N}}   % Formatting for N
\newcommand{\mK}{\mathbb{K}}   % Formatting for K
\newcommand{\mS}{\mathbb{S}}   % Formatting for S

\newcommand{\abs}[1]{\lvert #1 \rvert}  % Formatting for the absolute value
\newcommand{\norm}[1]{\lVert #1 \rVert}  % Formatting for the norm
\newcommand{\mD}{\mathscr{D}}

\newcommand{\mM}{\mathscr{M}}

\newcommand{\mF}{\mathscr{F}}

\newcommand{\mJ}{\mathcal{J}}

\newcommand{\bfi}{\mathbf{i}}
\newcommand{\bff}{\mathbf{f}}

\newcommand{\bfu}{\mathbf{u}}
\newcommand{\bfv}{\mathbf{v}}

\newcommand{\bfz}{\mathbf{z}}

\newcommand{\sfm}{\mathsf{m}}

\newcommand{\bmmu}{{\bm\mu}}

\DeclareMathOperator{\Bal}{Bal}

\DeclareMathOperator{\supp}{supp}

\newcommand{\rmd}{\mathrm{d}}

\makeatother

\usepackage{babel}

\usepackage{orcidlink}

\begin{document}

\title[Multiphase quadrature domains (existence and uniqueness)]{Multiphase quadrature domains\\ (existence and uniqueness)}

\begin{sloppypar}

\begin{abstract}  

The primary goal of this paper is to give a precise definition and prove existence and uniqueness of multiphase quadrature domains for subharmonic functions, ensuring that the prescribed measures are supported in the \emph{interior} of the resulting domains. The approach to prove existence is based on a variational framework, where we minimize an energy functional over so called segregated states. In this respect we refine earlier results in this direction. But we also show that this approach alone is not enough for two reasons. First of all it seems hard to get existence results which ensure that the interior support condition is satisfied. And second it may happen, as we show by an example, that a multiphase quadrature domain exists but is not a minimizer of the energy functional.

The main novelty of this work is the study of minimizers, if they exist, of the energy functional over a subset of the segregated states given by a natural constraint with respect to the given measures. From this approach we are able to prove uniqueness, and also give sufficient conditions for existence.
We also give an example showing that, unlike the energy minimization and partial balayage approaches which are equivalent in the one-phase case, this equivalence breaks down already in the two-phase setting.
\end{abstract}

\author[Kow]{Pu-Zhao Kow\,\orcidlink{0000-0002-2990-3591}}
\address{Department of Mathematical Sciences, National Chengchi University, Taipei 116, Taiwan}
\email{pzkow@g.nccu.edu.tw} 

\author[Shahgholian]{Henrik Shahgholian\,\orcidlink{0000-0002-1316-7913}}
\address{Department of Mathematics, KTH Royal Institute of Technology, SE-10044 Stockholm, Sweden %\& Yerevan State University, Armenia
}
\email{henriksh@kth.se}

\author[Sj{\"o}din]{Tomas Sj{\"o}din}
\address{Department of Mathematics, Link{\"o}ping University, SE-58183 Link{\"o}ping, Sweden}
\email{tomas.sjodin@liu.se}


\maketitle

\tableofcontents{}

\section{Introduction}

We work in the Euclidean space $\mR^n$, with $n \geq 2$ fixed throughout the paper. 
For any open set $D \subset \mR^{n}$, we denote by $\mM(D)$ the class of positive Radon measures on $\mR^{n}$ whose supports are compact and contained in $D$. 
The aim of this paper is to study so-called multiphase quadrature domains. We begin by recalling the classical (one-phase) notion.
A bounded open set $D$ is called a quadrature domain for subharmonic functions with respect to a measure $\mu\in\mM(D)$ if 
\begin{equation*}
\int_{D} s \, \rmd \sfm \geq \int s \,\rmd\mu, 
\end{equation*}
holds for every subharmonic function $s$ on $D$ that is integrable over $D$, with respect to the $n$-dimensional Lebesgue measure $\sfm$. 
We note that there are various notions of quadrature domains, depending on the class of test functions $s$ and the assumptions imposed on the measure $\mu$. 
Besides subharmonic functions, the most prominent examples are quadrature domains for harmonic or analytic functions. In those settings, the measure $\mu$ may even be replaced by a distribution (in the classical case, typically supported on a finite set of points).
In this paper, however, we restrict our attention to quadrature domains for subharmonic functions with positive measures.

The classical mean value inequality for subharmonic functions can be reformulated by observing that each ball $B_{R}=B_{R}(0)$ is a quadrature domain with respect to the measure $\mu=\frac{1}{\abs{B_{R}}}\delta_{0} \in \mM(B_{R})$, where $\delta_{0}$ denotes the Dirac delta supported at $0$. Quadrature domains has been extensively studied; we refer the reader to \cite{Gus04LecturesBalayage,GS05QuadratureDomain} and the references therein for details. We also note that this concept admits a natural extension to the Helmholtz equation, see \cite{GS24PartialBalayageHelmholtz,KLSS22QuadratureDomain,KSS23Minimization}.

We also recall the following PDE characterization of quadrature domains (see, e.g., \cite{Gus90QuadratureDomains} or \cite[Proposition~2.1]{KLSS22QuadratureDomain}). A bounded open set $D\subset\mR^{n}$ is a quadrature domain with respect to $\mu \in \mM(D)$ if and only if there is a distribution $u$ on $\mR^{n}$ such that 
\begin{equation}
-\Delta u = (\mu-1)\chi_{D} \text{ in $\mR^{n}$} ,\quad u \geq 0 \text{ with equality in $\mR^{n}\setminus D$.} \label{eq:1QD}
\end{equation}
A fundamental property of quadrature domains for subharmonic functions is that, for any given measure $\mu \in \mM(\mR^n)$ there exists at most one such quadrature domain, up to a set of Lebesgue measure zero. Moreover, such domains can be constructed via the partial balayage procedure, which we briefly review in \Cref{sec:partial-balayage} below. 
In general, it may happen that the function $u$ above vanishes at some points inside $D$. Roughly speaking, this phenomenon is responsible for the fact that uniqueness holds only up to sets of Lebesgue measure zero.
So, to get uniqueness, we need the additional assumption that we have a strict inequality in the PDE charactherization above, i.e. that $u>0$ in $D$. If this is satisfied we say that $D$ is a \emph{strong} quadrature domain for subharmonic functions.

Prior to this paper the case of quadrature domains with two phases has also been studied in \cite{GS12TwoPhaseQD}, which can be defined as follows. Given two disjoint bounded open sets $D_1,D_2$ and measures $\mu_1 \in \mM(D_1),\mu_2 \in \mM(D_2)$ then the pair $(D_1,D_2)$ is called a two-phase quadrature domain for subharmonic functions with respect to $(\mu_1,\mu_2)$ if there are functions $(u_1,u_2)$ such that $u_i \geq 0$ with equality in $\mR^n \setminus D_i$ and such that the difference $u_1-u_2$ satisfies
\begin{equation}
-\Delta (u_1-u_2)=(\mu_1-1)\chi_{D_1}-(\mu_2-1)\chi_{D_2} \textrm{ in } \mR^n.   
\end{equation}
Note in particular that along the common boundary $\partial D_1 \cap \partial D_2$ the difference $u_1-u_2$ is locally differentiable, and along this part we must have that the outward normal derivatives of $u_i$ with respect to $D_i$ (as long as they are well defined) are equal, because otherwise there would be an additional term in $-\Delta (u_1-u_2)$ along this set.

Also some prior work in a multi-phase setting has been conducted. Consider pairwise disjoint bounded open sets $Q_{1},\cdots,Q_{m}\subset\mR^{n}$, satisfying the following properties:
there exist measures $\mu_{1}\in\mM(Q_{1}),\cdots,\mu_{m}\in\mM(Q_{m})$ and non-negative functions $u_{1},\cdots,u_{m}$ with $Q_{i}=\{u_{i}>0\}$ such that 
\begin{equation}
-\Delta(u_{i}-u_{j})=(\mu_{i}-1)\chi_{Q_{i}}-(\mu_{j}-1)\chi_{Q_{j}} \quad\text{in $\mR^{n}\setminus\bigcup_{\ell\neq i,j}\overline{Q_{\ell}}$.} \label{eq:multiphase-PDE}
\end{equation}
In \cite{AS16MultiPhaseQD}, the authors constructed non-negative functions $u_{1},\cdots,u_{m}$ with $Q_{i}=\{u_{i}>0\}$ satisfying \eqref{eq:multiphase-PDE} by minimizing a suitable functional among all segregated states, and a generalization to the Helmholtz equation was studied in \cite{KS24MultiplaseQD}. However, these constructions do not guarantee the \emph{support condition}
\begin{equation}
\mu_{i} \in \mM(Q_{i}) ,\quad \text{for all $i=1,2,\cdots,m$.} \label{eq:support-condition}
\end{equation}
Some conditions and arguments to ensure $\supp\,(\mu_{i})\subset\overline{Q_{i}}$ has been considered, but the key point in \Cref{def:m-phase-QD} below is that the support $\supp\,(\mu_{i})$ is not allowed to reach the boundary $\partial Q_{i}$ of $Q_{i}$. It is worth noting that in the one-phase situation to guarantee that $\supp(\mu) \subset \overline{\{u>0\}}$ it is rather easy to see that it is enough that $\mu>1$ on its support, but even for measures which are singular with respect to Lebesgue measure $\supp(\mu) \subset \{u>0\}$ may fail. I.e. to get compact support is a much more subtle question. For instance we can mention the work done in \cite{Sa10SmallMod,GS14Stationary} which studies conditions to ensure that a corner in so called Hele-Shaw flow is immediately non-stationary, which is exactly the question of compact support for a specific type of measure of the form $\mu=t\delta_{x_0} + \sfm|_{D_0}$, where $D_0$, in the two-dimensional case, can be interpreted as an initial domain filled with a Newtonian fluid, and we have a source at $x_0 \in D_0$ where we continuously inject more fluid with time $t$, leading to a growing family of domains $D_t$. Then the question of compact support is simply in this setting if $\overline{D_0} \subset D_t$ for any $t>0$, which turns out to be rather involved to deal with even in this situation.

We now introduce strong multiphase quadrature domains for subharmonic functions. Henceforth when we write a sum of the form
\begin{equation*}
    \sum_{\ell \ne j} a_\ell
\end{equation*}
it is to be understood that $j$ is fixed and we sum over the set of $\ell$ in $\{1,2,\cdots,m\} \setminus \{j\}$, where $m$ will denote the number of phases, which should be understood from the context (note in particular that to sum over $\ell \ne j$ is not the same as to sum over $j\ne \ell$ with this convention, where we in the latter case should have $\ell$ fixed and sum over $j$ instead). We may also note that we have
\begin{equation*}
    \sum_{\ell \ne j} a_{\ell}-a_j=\left(\sum_{\ell \ne j} a_{\ell}\right)-a_j,
\end{equation*}
and if we need our coefficients to be $a_{\ell}-a_j$ in a sum this will be indicated by using parentheses $(a_{\ell}-a_j)$. 

\begin{definition}\label{def:m-phase-QD}
A collection of pairwise disjoint bounded open sets $Q_{1},Q_{2},\cdots,Q_{m}\subset\mR^{n}$ is called a \emph{strong $m$-phase quadrature domain} for subharmonic functions with respect to the measures $\mu_{1}\in\mM(Q_{1}),\mu_{2} \in \mM(Q_{2}),\cdots,\mu_{m}\in\mM(Q_{m})$  if there exist non-negative lower semicontinuous (LSC) functions $u_{1},u_{2},\cdots,u_{m}$ such that $Q_{i}=\{u_{i}>0\}$ for each $i$ and such that for any $j \in \{1,2,\cdots,m\}$ we have
\begin{equation}
-\Delta\left(\sum_{\ell \ne j} u_{\ell}-u_{j}\right) \leq \sum_{\ell \ne j}(\mu_{\ell}-1)\chi_{Q_{\ell}}-(\mu_{j}-1)\chi_{Q_{j}} \quad\text{in $\mR^{n}$.} \label{eq:multiphase-PDE-strong}
\end{equation} 
\end{definition}

In particular, when $m=1$, \Cref{def:m-phase-QD} coincides with the classical notion of a strong quadrature domain for subharmonic functions, and when $m=2$, it reduces to the notion of a strong two-phase quadrature domain.

Since strong $m$-phase quadrature domains for subharmonic functions are the only objects considered in this paper, we shall henceforth, unless otherwise specified, simply refer to them as $m$-phase quadrature domains.

Heuristically, for each pair $i \ne k$, we require $\nabla u_i = -\nabla u_k$ along their common boundary $(\partial Q_i \cap \partial Q_k)\setminus\bigcup_{\ell\neq i,k}\overline{Q_{\ell}}$. Moreover, at points on $\partial Q_i\setminus\bigcup_{\ell\neq i}\overline{Q_{\ell}}$, the gradient of $u_i$ should vanish, exactly as in the one-phase case.
To clarify, suppose $y \in (\partial Q_i \cap \partial Q_k)\setminus\bigcup_{\ell\neq i,k}\overline{Q_{\ell}}$ and let $U$ be an open neighborhood of $y$ that does not intersect $\bigcup_{\ell\neq i,k}\overline{Q_{\ell}}$. Then $u_\ell \equiv 0$ in $U$ for all $\ell \ne i,k$. Applying \eqref{eq:multiphase-PDE-strong} first with $j=i$ and then with $j=k$ gives 
\begin{equation}
-\Delta(u_{k}-u_{i}) = (\mu_{k}-1)\chi_{Q_{k}}-(\mu_{i}-1)\chi_{Q_{i}} \quad\text{in $U$.} 
\end{equation} 
Thus, in such a neighborhood $U$, the system reduces to the same equation that arises in the definition of two-phase quadrature domains.

By the local regularity theory for two-phase free boundary problems (see \cite[Theorem~6.12]{PSU12FreeBoundary}), within a neighborhood $V\subset U$ of $y$, both $\partial Q_i$ and $\partial Q_k$ are $C^1$ graphs and either meet tangentially at $y$ or coincide in some neighborhood of $y$.

In general, however, there may exist points at which three or more boundaries meet, and this necessitates some care in the formulation of the definition. To develop a theory that works in full generality, it is not sufficient to impose boundary conditions only on points where exactly two boundaries intersect, as in \eqref{eq:multiphase-PDE}. Indeed, even in the plane, there exist three (in fact, countably many) pairwise disjoint domains sharing the same boundary (so-called Wada's lakes) in which case such a definition would leave us with no effective boundary condition at all.

Nevertheless, provided that triple junctions (that is, points where three or more boundaries meet) are sufficiently rare (as is typically the case for reasonably regular configurations), the situation simplifies. In this setting, the boundary conditions derived from \eqref{eq:multiphase-PDE-strong} are, in essence, equivalent to requiring that the functions $u_i$ and $u_k$ satisfies $\nabla u_i = -\nabla u_k$ along $\partial Q_i \cap \partial Q_k$. Consequently, under this structural assumption, \eqref{eq:multiphase-PDE} and \eqref{eq:multiphase-PDE-strong} are effectively equivalent.

It is worth noting that, in the two-phase setting, taking the difference $u_1 - u_2$ cancels all boundary contributions in the Laplacian, leading to integral identities for appropriate harmonic functions \cite{EPS11TwoPhaseQD,GS12TwoPhaseQD}. For three or more phases, such a cancellation is no longer possible, which motivates the PDE-based generalization in our definition.

\begin{remark}\label{lem:average-L-infty}
Note that $\{Q_{i}\}_{i=1}^{m}$ is a strong $m$-phase quadrature domain with respect to $\{\mu_{i}\}_{i=1}^{m}$ if and only if it is also a strong $m$-phase quadrature domain with respect to $\{\mu_{i}^{\epsilon_{i}}\}_{i=1}^{m}$ for all $0<\epsilon_{i}<\frac{1}{2}{\rm dist}\,(\supp\,(\mu_{i}),\partial Q_{i})$ small enough, where 
\begin{equation*}
\mu_{i}^{\epsilon_i} = \mu_{i} * \frac{1}{\abs{B_{\epsilon_i}}}\chi_{B_{\epsilon_i}}. 
\end{equation*}
For completeness, we include a proof in \Cref{sec:partial-balayage}. 
In particular, this observation allows us to reduce the analysis to the case where the measures admit $L^{\infty}$ densities, a regularity assumption that is required for the energy minimization framework developed in the following. 
\end{remark}

The main goal of this paper is to further develop the existing minimization approach from \cite{AS16MultiPhaseQD,KS24MultiplaseQD} to prove existence and uniqueness results in the multi-phase setting. For the one and two-phase situations this has previously been handled by using (one- and two-phase) partial balayage, see  \cite{GS12TwoPhaseQD,GS24PartialBalayageHelmholtz,KS24MultiplaseQD}. This seems, due to a lack of monotonicity when we pass from two to three phases, not possible to generalize to three or more phases, and therefore we will instead focus on the energy minimization approach here (see \Cref{ex:tpqd-nonmin} below). The paper contains results regarding this that seem new also in the two-phase situation (see \Cref{ex:counterexample-tpqd-v3}).

First of all we have the following uniqueness theorem (which follows from  \Cref{prop:uniqueness-procedure} and \Cref{lem:average-L-infty}).
\begin{theorem} \label{thm:main-3}
There is at most one strong $m$-phase quadrature domain for subharmonic functions for any given collection of measures $\mu_1,\mu_2,\cdots,\mu_m$ in $\mM(\mR^n)$ with disjoint supports.      
\end{theorem}

When it comes to existence one can only expect to give sufficient conditions, which is even the case in the one-phase situation. We recall that, vaguely stated, in the one-phase situation the measure needs to be sufficiently concentrated on its support to ensure existence. This is still the case in the multi-phase situation (indeed as we will see below unless strong one-phase quadrature domains exists for each of the measures $\mu_{i}$ then there can't be a $m$-phase quadrature domain of this type with respect to $\mu_{1},\mu_{2},\cdots,\mu_{m}$ either), but additionally the measure $\mu_i$ must in some sense dominate over all the other measures $\mu_j$ close to its support.

In particular we will prove the following.

\begin{theorem}\label{thm:main-1-rough}
Let $\mu_{1},\mu_{2}\cdots,\mu_{m}$ in $\mM(\mR^n)$ 
where each measure satisfies the concentration condition
\begin{equation*}
    \limsup_{r \to 0_+} \frac{\mu_i(B_r(x))}{\sfm(B_r)} > 2^n \textrm{ for all } x \in \supp{(\mu_i)}.
\end{equation*}
Then there is for each $i$ a unique strong one-phase quadrature domain for subharmonic functions $\omega_i$ with respect to $\mu_i$. If, furthermore,  
\begin{equation*}
\supp(\mu_i) \cap \overline{\omega_j}=\emptyset \textrm{ for all } i \ne j,
\end{equation*}
then there is a strong $m$-phase quadrature domain for subharmonic functions with respect to $\mu_1,\mu_2,\cdots,\mu_m$.
\end{theorem} 

In many respects, the strongest existence result of this paper is contained in \Cref{cor:main-exist-cor}, but since its formulation is somewhat involved,  here we  just mention the following theorem, which is a rather simple consequence of this corollary. 

\begin{theorem}\label{thm:main-2}
 If each $\mu_{i}\ge 0$, $i=1,2,\cdots,m$, is a finite nontrivial  sum of point masses, and they have disjoint supports, then there is a strong $m$-phase quadrature domain for subharmonic functions with respect to $\mu_1,\mu_2,\cdots,\mu_m$. 
\end{theorem}

\section{Partial balayage\label{sec:partial-balayage}} 
% \addtocontents{toc}{\SkipTocEntry}

For the reader's convenience, we recall some terminology and results about the so called partial balayage, see for instance  \cite{GS94PartialBayage,GS12TwoPhaseQD,GS24PartialBalayageHelmholtz}, see also \cite{KLSS22QuadratureDomain}. Throughout, we assume $n\ge 2$ without further mention. Given an open set $D\subset\mR^{n}$ and a positive measure $\mu$ with compact support in $\mR^{n}$ and a non-negative density function $\rho \in L^{\infty}(\mR^n)$ such that $\rho \geq {\rm constant} >0$ outside a compact set, we define 
\begin{equation*}
\mF_{D,\rho}(\mu) := \left\{ v \in \mD'(\mR^{n}) : 
\begin{aligned}
& \text{$-\Delta v\le \rho$ in $D$} ,\quad \text{$v\le U^{\mu}$ in $\mR^{n}$} \\ 
\end{aligned}
\right\},
\end{equation*}
where $U^{\mu}$ denotes the Newtonian potential of $\mu$ (see, e.g., \cite[(4.1)]{GT01Elliptic}), defined by 
\begin{equation}
U^{\mu}(x):=\int\Psi(|x-y|)\,\rmd\mu(y) ,\quad \Psi(r)=\left\{\begin{aligned}
& -\frac{1}{2\pi}\log\abs{r}, && n=2, \\ 
& \frac{1}{n(n-2)\abs{B_1}\abs{r}^{n-2}}, && n\ge 3, 
\end{aligned}\right. \label{eq:fund}
\end{equation}
where 
\begin{equation}
\text{$\abs{B_1} = \frac{2\pi^{n/2}}{n\Gamma(n/a)}$ is the volume of the unit ball in $\mR^{n}$.} \label{eq:volume-unit-ball}
\end{equation}
Note that, in the sense of distributions, 
\begin{equation*}
-\Delta U^{\mu} = \mu.
\end{equation*}
Standard potential-theoretic arguments \cite[Section~3.7]{AG01Potential} show that $\mF_{D,\rho}(\mu)$ has a largest element, which furthermore admits a continuous representative $V_{D,\rho}^{\mu}$. The function $V_{D,\rho}^{\mu}$ is also called the \emph{partial reduction} of $U^{\mu}$ \cite{GS09PartialBalayage} with respect to $\rho$ over $D$. 
In many respects the most important function for us will be
\begin{equation*}
    W_{D,\rho}^{\mu}=U^{\mu}-V_{D,\rho}^{\mu},
\end{equation*}
and we also define the \emph{non-contact set} by 
\begin{equation*}
\omega_{D,\rho}(\mu) := \{V_{D,\rho}^{\mu}<U^{\mu}\} = \{W_{D,\rho}^{\mu}>0\}, 
\end{equation*}
which, up to a polar subset of $\partial D$, is known to always be a bounded open subset of $D$ under the above assumptions. Also note that $W_{D,\rho}^{\mu}$ is the smallest  function satisfying
\begin{equation}
    W_{D,\rho}^{\mu} \ge 0 \textrm{ in } \mR^n \textrm{ and } -\Delta W_{D,\rho}^{\mu} \ge \mu-\rho \textrm{ in } D, \label{eq:obstacle-problem}
\end{equation}
(or, more precisely, a specific representative of the smallest distribution satisfying the above in the sense of distributions).
The following is furthermore known to hold 
\begin{align}
&W_{D,\rho}^{\mu} \le W_{D,\rho}^{\eta} \quad \text{and} \quad \omega_{D,\rho}(\mu) \subset \omega_{D,\rho}(\eta) \textrm{ if } \mu \leq \eta, \label{eq:monotonicity-in-measure-partial-reduction}\\
&W_{D_{1},\rho}^{\mu} \ge W_{D_{2},\rho}^{\mu} \quad \text{and} \quad \omega_{D_{2},\rho}(\mu) \subset \omega_{D_{1},\rho}(\mu)\cap D_{2} \quad \text{for all open sets $D_{1}\supset D_{2}$.} \label{eq:monotonicity-partial-reduction}
\end{align} 
The \emph{partial balayage} is now defined by 
\begin{equation*}
\Bal_{D,\rho}(\mu) := -\Delta V_{D,\rho}^{\mu}, 
\end{equation*}
where the above again should be interpreted in the sense of distributions, and it is then always a positive measure under our assumptions.

Using $V_{D,\rho}^{\mu}=U^{\mu}$ in $\mR^{n}\setminus \omega_{D,\rho}(\mu)$, one can easily check that $V_{D,\rho}^{\mu}=U^{\Bal_{D,\rho}(\mu)}$. The partial balayage also satisfies the following structural decomposition: 
\begin{equation}
\Bal_{D,\rho}(\mu) = \rho|_{\omega_{D,\rho}(\mu)} + \mu|_{\mR^{n}\setminus\omega_{D,\rho}(\mu)} + \nu, \label{eq:structure-partial-balayage}
\end{equation}
where $\nu\ge 0$ is supported on $\partial D\cap \partial \omega_{D,\rho}(\mu)$, and $\Bal_{D,\rho}(\mu) \le \rho$ in $D$. 

When $D=\mR^n$ we drop it from the notation and write e.g. $V_{\rho}^{\mu}=V_{\mR^n,\rho}^{\mu}$ (and similarly for the other entities above).
Furthermore, in case $\rho=1$ then we will drop $\rho$ from the notation and simply  write $V_{D}^{\mu}=V_{D,1}^{\mu}$. In particular when $D=\mR^n$ and $\rho=1$, which is the case connected to quadrature domains, we drop both from the notation and write for instance $V^{\mu}=V_{\mR^n,1}^{\mu}$ with these conventions.

In case $\mu$ belongs to $L^{\infty}(\mR^n)$, then one can alternatively get $W_{D,\rho}^{\mu}$ from the following energy minimization approach, which is what we will work with for the case of multiple phases below. To do so let us introduce the function $f=\mu-\rho$, the class $\mK(D) =\{v \in H^1_0(D): v \ge 0\}$ and for $v \in H^1_0(D)$ the functional
\begin{equation}\label{eq:definition-of-onephase-J_f}
J_f(v) = \int_D |\nabla v|^2 \,\rmd \sfm - 2\int_D vf \,\rmd \sfm. 
\end{equation}

Then, for any $v \in \mK(D)$ we have
\begin{align*}
J_f(v) &= \int_D |\nabla v|^2 \,\rmd \sfm - 2\int_D vf \,\rmd \sfm  \overset{\eqref{eq:obstacle-problem}}{\ge}  \int_D |\nabla v|^2 \,\rmd \sfm - 2\int_D v (-\Delta W_{D,\rho}^{\mu}) \,\rmd \sfm \\
 &= \int_D |\nabla v|^2 \,\rmd \sfm - 2\int_D \nabla v \cdot \nabla W_{D,\rho}^{\mu} \,\rmd \sfm \ge -\int_D |\nabla W_{D,\rho}^{\mu}|^2 \,\rmd \sfm =J_f(W_{D,\rho}^{\mu}).
\end{align*}
Hence, by the above and strict convexity, we see that 
\begin{equation}\label{eq:minimizerJf}
W_{D,\rho}^{\mu} \textrm{ is the unique minimizer of }J_f \textrm{ over } \mK(D).
\end{equation}
When $D=\mR^n$ we again drop it from the notation, and then we may note that $\mK=H^1_0(\mR^n)=H^1(\mR^n)$.

In general, whether  $\mu \in \mM(D)$ has compact support inside $\omega_D(\mu)$ or not depends on $D$. However, if $\mu$ is sufficiently concentrated this is always the case. To understand this, suppose first that $\mu$ has a point mass at $x$ (i.e. $\mu (\{x\})>0$) and $D$ contains $x$, then for a small enough number $\varepsilon>0$ we have that both $\varepsilon \delta_x \le \mu$ and $\omega(\varepsilon \delta_x) \subset D$. Then it is easy to see that
$\omega_D(\mu) \supset \omega_D(\varepsilon \delta_x)=\omega(\varepsilon \delta_x)$, and the latter is an open ball centered at $x$. More generally  the following result holds (see for instance the proof of  \cite[Corollary~5.2]{GS12TwoPhaseQD}).
\begin{lemma}\label{lem:concentration-condition-one-phase}
Let $\mu \in \mM(\mR^n)$.
\begin{enumerate} 
\renewcommand{\labelenumi}{\theenumi}
\renewcommand{\theenumi}{\rm (\alph{enumi})} 
\item \label{itm:a} If $\supp(\mu) \subset B_r(x)$ and $\mu(B_r(x)) >2^n\sfm(B_r)$, then $\overline{B_r(x)} \subset \omega(\mu)$.
\item \label{itm:b} If $\mu$ satisfies the concentration condition 
    \begin{equation*}
        \limsup_{r \to 0_+} \frac{\mu(B_r(x))}{\sfm(B_r)} > 2^n \textrm{ for all } x \in \supp(\mu)
    \end{equation*}
    then for any open set $D$ which contains $\supp(\mu)$ we also have $\supp(\mu) \subset \omega_D(\mu)$.
\end{enumerate} 
\end{lemma}

\begin{remark}\label{lem:con-mqd-pb}
We remark that there is  a direct connection between this one-phase construction and the multi-phase quadrature domains introduced above.
If $u_1,u_2,\cdots,u_m$ are the functions appearing in the definition of a strong $m$-phase quadrature domain for subharmonic functions with respect to $\mu_1,\mu_2,\cdots,\mu_m$, then it is straightforward to see that 
\begin{equation*}
u_i=W_{Q_i}^{\mu_i} \text{ in $\mR^{n}$} \quad \text{for all $i=1,\cdots,m$,}
\end{equation*}
therefore, each domain $Q_{i}$ is invariant under the mapping $Q\mapsto\omega_{Q}(\mu_{i})$, that is, 
\begin{equation}
Q_{i} = \omega_{Q_{i}}(\mu_{i}) \quad \text{for all $i=1,\cdots,m$.} \label{eq:fixed-point}
\end{equation}
To see this, we note that by definition $u_i \ge W_{Q_i}^{\mu_i}$, but also $-\Delta u_i \le -\Delta W_{Q_i}^{\mu_i}$. And since both functions are $0$ on $\partial Q_i$ they must hence be equal by the maximum principle.

\end{remark}

In view of \Cref{lem:con-mqd-pb}, we observe that \Cref{lem:average-L-infty} follows immediately from the following lemma. 

\begin{lemma}\label{lem:stability-mollification}
    If $\supp(\mu) \subset \omega_{D,\rho}(\mu)$, then $\supp(\mu^{\epsilon}) \subset \omega_{D,\rho}(\mu)=\omega_{D,\rho}(\mu^{\epsilon})$ and $V_{D,\rho}^{\mu^{\epsilon}} = V_{D,\rho}^{\mu}$ for all $\epsilon>0$ small enough.
\end{lemma}
\begin{proof}
First of all it is immediate that $\supp(\mu^{\epsilon}) \subset \omega_{D,\rho}(\mu)$ as long as $\epsilon<{\rm dist}(x,\partial \omega_{D,\rho}(\mu))$, and below we assume that $\epsilon$ always satisfies this.
    By the mean-value inequality for superharmonic functions, it is well known that $U^{\mu^{\epsilon}} \nearrow U^{\mu}$ as $\epsilon \searrow 0$, and furthermore we have equality $U^{\mu^{\epsilon}} (x) = U^{\mu}(x)$ for all $x$ such that ${\rm dist}(x,\supp(\mu))>\epsilon$. By assumption $V_{D,\rho}^{\mu} < U^{\mu}$ in a neighborhood of $\supp(\mu)$, and therefore it is easy to see that we may choose $\epsilon$ small enough such that $V_{D,\rho}^{\mu} < U^{\mu^{\epsilon}}$ in a neighborhood of $\supp(\mu^{\epsilon})$ as-well. But then it is more or less immediate by definition that $V_{D,\rho}^{\mu^{\epsilon}}=V_{D,\rho}^{\mu}$ for such $\epsilon$, and that $W_{D,\rho}^{\mu^\epsilon} =U^{\mu^{\epsilon}}-V_{D,\rho}^{\mu}$. Furthermore, $W_{D,\rho}^{\mu^\epsilon}$ is positive on $\omega_{D,\rho}(\mu)$, so it follows that $\omega_{D,\rho}(\mu^{\epsilon}) = \omega_{D,\rho}(\mu)$.
\end{proof}

\section{A minimization problem} 
The goal of this section is to further develop the energy minimization approach of \cite{AS16MultiPhaseQD} and then use this to prove our existence and uniqueness results. As we noted above for the one-phase situation, as long as we have measures with $L^{\infty}$ densities, then we could use either the partial balayage method or the energy minimization method. 
For the case of two phases first studied in \cite{EPS11TwoPhaseQD} the authors used the method of energy minimization, and then in \cite{GS12TwoPhaseQD} a type of two-phase partial balayage method was introduced. A particular merit of the latter was that it could be used to prove uniqueness and give sufficient conditions to guarantee existence of two-phase quadrature domains, which are the main aims of this article to get in the multi-phase case. We start this section with an example, based on the method of two-phase partial balayage, which in particular indicates why we think that an approach to develop an $m$-phase partial balayage along similar lines as the two-phase version seems out of reach.

\begin{example}\label{ex:tpqd-nonmin}
In the definition of two-phase partial balayage, it is crucial that a certain monotonicity property holds when passing from the one-phase to the two-phase setting. 
For each signed measure $\mu$, we denote by $\overline{W}^{\mu}$ the smallest function in the class 
\begin{equation}
\tau _{\mu }:=\{w:w\text{ is }\delta \text{-subharmonic, }-\Delta w\geq \eta
(w,\mu )\text{ and }w\geq -W^{\mu ^{-}}\text{ on }\mathbb{R}^{n}\}, \label{eq:def-2-phase-balayage}
\end{equation}
where 
\begin{equation*}
\eta (w,\mu )=\left( (\mu ^{+}-\sfm )^{+}-(\mu ^{+}-\sfm
)^{-}|_{\{w>0\}}\right) -\left( (\mu ^{-}-\sfm )^{+}-(\mu ^{-}-\sfm
)^{-}|_{\{w<0\}}\right) .
\end{equation*} 
In particular, the condition $w \ge -W^{\mu^-}$ serves as a penalization constraint for functions $w \in \tau_\mu$. For later reference, we denote 
\begin{equation*}
\Omega^+(\mu) = \{\overline{W}^\mu >0\} ,\quad \Omega^-(\mu) = \{\overline{W}^\mu < 0\}.
\end{equation*}
Note in particular that the penalization implies that $\Omega^-(\mu) \subset \omega(\mu^-)$.

Let us consider the case of three point masses in the plane:
$$\mu_1=\frac{4\pi}{9}\delta_{(-1,0)}, \quad \mu_2=\frac{4\pi}{9}\delta_{(0,0)}, \quad \mu_3=\frac{4\pi}{9}\delta_{(1,0)}.$$
The one-phase partial balayage of each of these onto $\rho\equiv 1$ will be a ball of radius $r=2/3$ centered at the respective points $(-1,0),(0,0)$ and $(1,0)$. 
Let us now introduce the measure 
\begin{equation*}
\nu=\mu_1+\mu_3-\mu_2. 
\end{equation*} 
According to \cite[Corollary~5.4]{GS12TwoPhaseQD}, the two-phase partial balayage of $\nu$ is well defined and produces a two-phase quadrature domain $\left(\Omega^+(\nu),\Omega^-(\nu)\right)$. 
The set $\Omega^+(\nu)$ splits into two disjoint ``ball-like'' components $D_1$ and $D_3$, each surrounding the support of the measures $\mu_1$ and $\mu_3$, respectively, while the set $\Omega^-(\nu)$ consists of a single component $D_{2}$. It is then straightforward to verify that the triple $(D_{1},D_{2},D_{3})$ forms a three-phase quadrature domain for the measures $\mu_{1},\mu_{2},\mu_{3}$, taking advantage of the fact that the one-phase quadrature domains of $\mu_{1}$ and $\mu_{3}$ are disjoint. 

The construction of two-phase partial balayage, which is used above, works because the transition from one phase to two phases is monotonic in the sense that $\Omega^+(\mu) \subset \omega(\mu^+)$ and $\Omega^{-}(\mu) \subset \omega(\mu^-)$, which made it possible to use the constraint in the definition of the class $\tau_{\mu}$ above. The same property holds when passing from one to multiple phases (see inequality \eqref{eq:monotonicity-support} below) but we will now explain that this does not hold in the same way when passing from two to three phases, so it would not be possible to pass from two to three phases in such a way that we first construct a two-phase quadrature domain $Q_1,Q_2$ with respect to $\mu_1,\mu_2$, and then try to construct a three phase quadrature domain $D_1,D_2,D_3$ for $\mu_1,\mu_2,\mu_3$ with any construction which forces $D_1 \subset Q_1$ and $D_2 \subset Q_2$.

Using \cite[Lemma~4.1(c), Theorem~4.3(a)]{GS12TwoPhaseQD} and the fact that $\mu\le\nu$, it is straightforward to see that $\tau_{\nu}\subset\tau_{\mu}$. This immediately implies $\overline{W}^\nu \geq \overline{W}^\mu$, and consequently, 
$$\Omega^+(\mu) \subset \Omega^+(\nu) \quad \text{and} \quad \Omega^-(\nu) \subset \Omega^-(\mu).$$
In other words, the set where the function is positive becomes larger, while the set where it is negative becomes smaller. This relation also holds at the level of individual components: the first component $D_1$ above must be larger than $\Omega^+(\mu)$, and in fact strictly larger, as follows from Hopf's lemma. On the other hand, $D_2$ will be strictly contained within $\Omega^-(\mu)$. This illustrates that monotonicity fails when moving from two to three phases, making it difficult to impose such constraints iteratively on the admissible functions in a natural way. 
\end{example}

One may naturally ask to what extent the energy minimization method and the two-phase partial balayage method are equivalent, as they are in the one-phase case. This question is especially relevant because the energy minimization framework can be extended naturally to the multiphase setting \cite{AS16MultiPhaseQD,KS24MultiplaseQD}. Unfortunately, this is not the case, as we explain at the end of this section in \Cref{ex:counterexample-tpqd-v3}.

Let $\mu_1,\mu_2,\cdots,\mu_m$ and $\rho_1,\rho_2,\cdots,\rho_m$ be non-negative functions in $L^{\infty}(\mR^n)$ such that, for each $j=1,2,\cdots,m$, $\mu_j$ has compact support and  $\rho_j>c$ outside a compact set for some constant $c>0$ . Furthermore define 
\begin{equation}\label{eq:assumption-A-rho}
f_j=\mu_j-\rho_j \textrm{ for each  } j=1,2,\cdots,m,
\end{equation}
and
\[\bff=(f_{1},\cdots,f_{m}).\]
Note  that
\begin{equation}
\text{$f_{j}\le {\rm constant} < 0$ holds outside a compact set}, \label{eq:assumption-A} 
\end{equation}
which in particular implies that $\supp\,(f_{j}^{+})$ is compact (and vice versa any $f_j \in L^{\infty}(\mR^n)$ satisfying this can be written on the above form).
The case relevant for $m$-phase quadrature domains is when each $\rho_j=1$, so that 
\begin{equation}
f_{j} = \mu_{j} - 1 \quad \text{for all $j=1,\cdots,m$.} \label{eq:special-choice-f}
\end{equation}
The results up to \Cref{cor:cont-min-Sm} below will be proved in the more general setting, but after this we assume that $\rho_j=1$ unless otherwise stated.

We now introduce the functional
\begin{equation}
\mJ_{\bff}(\bfu) := \sum_{i=1}^{m} J_{f_{i}}(u_{i}) = \sum_{i=1}^{m} \int_{\mR^{n}} \left( \abs{\nabla u_{i}}^{2} - 2f_{i}u_{i} \right) \,\rmd \sfm \label{eq:functional-J-definition}
\end{equation}
for $\bfu = (u_{1},\cdots,u_{m})\in(H^{1}(\mR^{n}))^{m}$. 
Define 
\begin{equation*}
\begin{aligned}
\mK^{m} &:= \left\{ \bfu \in (H^{1}(\mR^{n}))^{m} : \text{$u_{i}\ge 0$ for all $i=1,\cdots,m$} \right\}, \\ 
\mS_{m} &:= \left\{ \bfu \in \mK^{m} : \text{$u_{i}\cdot u_{j}=0$ for all $i\neq j$} \right\}. 
\end{aligned}
\end{equation*}
Following \cite{CTV05SegregationProblem}, we refer to the elements of $\mS_{m}$ as \emph{segregated states}, where
a state $\bfu \in (H^{1}(\Omega))^{m}$ is called \emph{segregated} if $u_{i}\cdot u_{j}=0$ for all $i\neq j$.

It should be pointed out, that although we use some terminology and ideas from \cite{CTV05SegregationProblem}, our problem is fundamentally different and does not fall under the scope of that paper, where we need other arguments and ideas. 

\begin{remark}\label{rem:quasicont-rep}
Given a segregated state $(u_1,u_2,\ldots,u_m) \in \mS_m$, the sets $Q_j=\{u_j>0\}$ will be of fundamental importance for us. As an element in $L^{\infty}$ the functions $\chi_{\{u_j>0\}}$ are well defined, since two representatives are equal almost everywhere (we will use the abbreviation ``a.e.'' in what follows), but the sets $\{u_j>0\}$ of-course depend on the specific choice of representatives of the functions $u_1,u_2,\ldots,u_m$. The condition $u_{\ell} \cdot u_j=0$ if $\ell \ne j$ means that this holds a.e. for any choice of representatives. We may always choose quasicontinuous\footnote{This means that the pre-image of every open set is quasi-open.} (and everywhere non-negative) representatives, and then if $j \ne \ell$, $u_j=0$ actually holds quasi-everywhere (we will use the abbreviation ``q.e.'' in what follows, see, e.g., \cite[Definition~5.1.10]{AG01Potential} for a precise definition) on $\overline{\{u_{\ell}>0\}}$.
This follows because $Q_{\ell}=\{u_{\ell}>0\}$ is quasi-open\footnote{This means that the set is contained in the closure of its interior.}, and we have $u_j=0$ q.e. in this set. Furthermore, the set of points in $\overline{Q_{\ell}}$ where $Q_{\ell}$ is thin has capacity zero, and it is easy to see that if there were a set of positive capacity of $\partial Q_{\ell}$ where $u_j>0$, then this contradicts that $u_j$ is quasicontinuous. We may therefore also, if necessary, redefine $u_j$ to be zero on all of $\overline{\{u_{\ell}>0\}}$ without breaking the quasicontinuity.
 \end{remark}

If we for given non-negative functions $\mu_i,\rho_i$ in $L^{\infty}(\mR^n)$, where $\mu_i$ has compact support and  $\rho_i>\textrm{constant}>0$ holds outside a compact set, define 
\begin{equation}
f_i=\mu_i-\rho_i,
\end{equation}
then  \eqref{eq:assumption-A} above is satisfied. Conversely, if \eqref{eq:assumption-A} is satisfied, then $f_i$ may always be written on the form \eqref{eq:assumption-A-rho}.  This is the case, with $\rho_i=1$, which is relevant for quadrature domains.

From the results in the previous section we know that, in case (\ref{eq:assumption-A-rho}) is satisfied, $v_{*,i}=W_{\rho_i}^{\mu_i}$ is the unique minimizer of $J_{f_i}$ over $\mK$, and since minimizing $\mJ_{\bff}$ over $\mK^m$ simply means that we minimize each $J_{f_i}$ over $\mK$ separately, we see in particular that  $\bfv_{*}=(W_{\rho_1}^{\mu_1},W_{\rho_2}^{\mu_2},\cdots,W_{\rho_m}^{\mu_m})$ is the unique minimizer of $\mJ_{\bff}$ over $\mK^{m}$, so we get the following result. 
\begin{equation}
\text{There is a unique minimizer $\bfv_{*}$ of $\mJ_{\bff}$ over $\mK^m$, and it has compact support.} \label{eq:boundedness}
\end{equation}

We now consider the minimization problem
\begin{equation*}
\text{minimize $\mJ_{\bff}(\bfu)$ over the segregated states $\bfu\in\mS_{m}$.}
\end{equation*}
By \cite[Theorem~1]{AS16MultiPhaseQD}, if \eqref{eq:assumption-A} holds, there is at least one minimizer $\bfu_{*}$ of $\mJ_{\bff}$ in $\mS_{m}$, and all minimizers have compact support. Moreover, for each minimizer $\bfu_{*}$, 
\begin{equation}
u_{*,i} \le v_{*,i} \text{ in $\mR^{n}$} \quad \text{for all $i=1,\cdots,m$} \label{eq:monotonicity-support}
\end{equation}
where $\bfv_{*}=(v_{*,1},\cdots,v_{*,m})$ is the unique minimizer of $\mJ_{\bff}$ over $\mK^{m}$. 
We now present a refinement of \cite[Proposition~1]{AS16MultiPhaseQD} in the following lemma. 

\begin{lemma}\label{lem:improvement-prop1-AS16}
Let $\bff=(f_{1},\cdots,f_{m})$ where each function $f_{1},\cdots,f_{m} \in L^{\infty}(\mR^{n})$ satisfies \eqref{eq:assumption-A}. If $\bfu_{*}=(u_{*,1},\cdots,u_{*,m})$ is a minimizer of $\mJ_{\bff}$ over $\mS_{m}$, then for each $j=1,\cdots,m$ we have 
\begin{equation} 
-\Delta\left(\sum_{\ell\neq j}u_{*,\ell}-u_{*,j}\right) \le \sum_{\ell \neq j}f_{\ell}\chi_{\{u_{*,\ell}>0\}} - f_{j}\chi_{\{u_{*,j}>0\}} \quad \text{in $\mR^{n}$.} \label{eq:Euler-Lagrange-general}
\end{equation}
\end{lemma}

\begin{proof}[Proof of \Cref{lem:improvement-prop1-AS16}]
Let $\psi\in C_{c}^{\infty}(\mR^{n})$ be non-negative. For $\epsilon>0$, define $\bfz=(z_{1},\cdots,z_{m})$ by 
\begin{equation*}
\left\{\begin{aligned}
z_{j} & := \left( \sum_{\ell\neq j}u_{*,\ell} - u_{*,j} - \epsilon\psi \right)^{-}, \\ 
z_{\ell} & := (u_{*,\ell} - u_{*,j} - \epsilon\psi)^{+} && \text{for all $\ell\neq j$.}
\end{aligned}\right.
\end{equation*}
Observe that 
\begin{equation*}
\Omega_{\epsilon} := \{z_{j}>0\}  =  \left\{ \sum_{\ell\neq j}u_{*,\ell} - u_{*,j} - \epsilon\psi < 0 \right\} \subset \bigcap_{\ell\neq j} \{z_{\ell}=0\}, 
\end{equation*}
$\mR^{n}\setminus\Omega_{\epsilon}\subset\{u_{*,j}=0\}$ and $0 \le \mJ_{\bff}(\bfz) - \mJ_{\bff}(\bfu_{*}) = I_{1}+I_{2}$ with 
\begin{equation*}
\begin{aligned}
I_{1} &= \sum_{\ell=1}^{m}\int_{\mR^{n}}(\abs{\nabla z_{\ell}}^{2} - \abs{\nabla u_{*,\ell}}^{2})\,\rmd \sfm \\ 
I_{2} &= -\sum_{\ell=1}^{m} \int_{\mR^{n}} 2f_{\ell}(z_{\ell}-u_{*,\ell})\,\rmd \sfm. 
\end{aligned}
\end{equation*}
We begin by computing $I_{1}$. Since $u_{*,\ell}\cdot u_{*,j}=0$ for all $\ell\neq j$, then 
\begin{equation*}
\begin{aligned}
I_{1} &= - \sum_{\ell\neq j} \int_{\Omega_{\epsilon}} \abs{\nabla u_{*,\ell}}^{2}\,\rmd \sfm + \int_{\Omega_{\epsilon}}\left( \left|\nabla\left(u_{*,j}+\epsilon\psi-\sum_{\ell\neq j}u_{*,\ell}\right)\right|^{2} - \abs{\nabla u_{*,j}}^{2} \right) \,\rmd \sfm \\ 
& \quad + \sum_{\ell\neq j} \int_{\mR^{n}\setminus \Omega_{\epsilon}}(\abs{\nabla(u_{*,\ell}-u_{*,j}-\epsilon\psi)}^{2} - \abs{\nabla u_{*,\ell}}^{2})\,\rmd \sfm - \overbrace{\int_{\mR^{n}\setminus \Omega_{\epsilon}} \abs{\nabla u_{*,j}}^{2}\,\rmd \sfm}^{=\,0} \\ 
& = - \sum_{\ell\neq j} \int_{\Omega_{\epsilon}} \abs{\nabla u_{*,\ell}}^{2}\,\rmd \sfm + 2\epsilon \overbrace{\int_{\Omega_{\epsilon}} \nabla\psi\cdot\nabla u_{*,j} \,\rmd \sfm}^{\int_{\mR^{n}} \nabla\psi\cdot\nabla u_{*,j} \,\rmd \sfm} + \int_{\Omega_{\epsilon}} \left|\nabla\left(\epsilon\psi-\sum_{\ell\neq j}u_{*,\ell}\right)\right|^{2}\,\rmd \sfm \\ 
& \quad - 2\epsilon \int_{\mR^{n}\setminus\Omega_{\epsilon}} \nabla\psi\cdot\nabla \left(\sum_{\ell\neq j}u_{*,\ell}\right)\,\rmd \sfm +  \overbrace{\sum_{\ell\neq j}\int_{\mR^{n}\setminus\Omega_{\epsilon}} \abs{\nabla(u_{*,j}+\epsilon\psi)}^{2} \,\rmd \sfm}^{=\,(m-1)\epsilon^{2}\int_{\mR^{n}\setminus\Omega_{\epsilon}} |\nabla \psi|^{2}\,\rmd \sfm} \\ 
& = 2\epsilon \int_{\mR^{n}} \nabla\psi\cdot\nabla u_{*,j} \,\rmd \sfm -2\epsilon\int_{\mR^{n}}\nabla\psi\cdot\nabla\left(\sum_{\ell\neq j}u_{*,\ell}\right)\,\rmd \sfm \\ 
& \quad + \epsilon^{2}\int_{\Omega_{\epsilon}}\abs{\nabla\psi}^{2}\,\rmd \sfm + (m-1)\epsilon^{2}\int_{\mR^{n}\setminus\Omega_{\epsilon}}\abs{\nabla\psi}^{2}\,\rmd \sfm. 
\end{aligned}
\end{equation*}
We proceed to compute $I_{2}$. Since $u_{*,\ell}\cdot u_{*,j}=0$ for all $\ell\neq j$, then 
\begin{equation*}
z_{\ell}-u_{*,\ell} = \left\{\begin{aligned}
& -\epsilon\psi && \text{if $u_{*,\ell}\ge\epsilon\psi$} \\ 
& -u_{*,\ell} && \text{if $u_{*,\ell}<\epsilon\psi$} 
\end{aligned}\right. \quad \text{for all $\ell\neq j$}
\end{equation*}
and 
\begin{equation*}
z_{j}-u_{*,j}=\left\{\begin{aligned}
& \epsilon\psi && \text{if $u_{*,j}>0$,} \\ 
& \epsilon\psi && \text{if $u_{*,j}=0$ and $\sum_{\ell\neq j}u_{*,\ell}=0$,} \\ 
& \epsilon\psi - \sum_{\ell\neq j}u_{*,\ell} && \text{if $0<\sum_{\ell\neq j}u_{*,\ell}<\epsilon\psi$,} \\ 
& 0 && \text{if $\sum_{\ell\neq j}u_{*,\ell}\ge\epsilon\psi$.}
\end{aligned}\right. 
\end{equation*}
Hence, we obtain 
\begin{equation*}
\begin{aligned}
I_{2} &= -2\sum_{\ell\neq j}\int_{\mR^{n}}f_{\ell}(z_{\ell}-u_{*,\ell})\,\rmd \sfm - 2\int_{\mR^{n}} f_{j}(z_{j}-u_{*,j})\,\rmd \sfm \\ 
& = 2\epsilon\sum_{\ell\neq j}\int_{\{u_{*,\ell}\ge\epsilon\psi\}}f_{\ell}\psi\,\rmd \sfm + 2\epsilon\sum_{\ell\neq j}\int_{\{u_{*,\ell}<\epsilon\psi\}}f_{\ell}u_{*,\ell}\,\rmd \sfm \\ 
& \quad -2\epsilon\int_{\{u_{*,j}>0\}}f_{j}\psi\,\rmd \sfm - 2\epsilon\int_{\{u_{*,j}=\sum_{\ell\neq j}u_{*,\ell}=0\}}f_{j}\psi\,\rmd \sfm - 2\epsilon\int_{\{0<\sum_{\ell\neq j}u_{*,\ell}<\epsilon\psi\}}f_{j}\psi\,\rmd \sfm \\ 
& \quad + 2\int_{\{0<\sum_{\ell\neq j}u_{*,\ell}<\epsilon\psi\}}f_{j}\sum_{i\neq j}u_{*,\ell}\,\rmd \sfm .
\end{aligned}
\end{equation*}
Consequently, we have 
\begin{equation*}
\begin{aligned}
0 &\le \limsup_{\epsilon\rightarrow 0_{+}}\frac{1}{2\epsilon}(I_{1}+I_{2}) \\ 
&\le \int_{\mR^{n}} \nabla\left( u_{*,j}-\sum_{\ell\neq j}u_{*,\ell}\right)\cdot\nabla\psi\,\rmd \sfm \\
& \quad + \sum_{\ell\neq j}\int_{\{u_{*,\ell}>0\}}f_{\ell}\psi\,\rmd \sfm - \int_{\{u_{*,j}>0\}}f_{j}\psi\,\rmd \sfm - \int_{\cap_{\ell}\{u_{*,\ell}=0\}} f_{j}\psi\,\rmd \sfm.
\end{aligned}
\end{equation*}
Since $\psi$ was arbitrary, we now obtain, in the sense of distributions, 
\begin{equation*}
-\Delta\left(\sum_{\ell\neq j}u_{*,\ell}-u_{*,j}\right) \le \sum_{\ell\neq j}f_{\ell}\chi_{\{u_{*,\ell}>0\}} - f_{j}\chi_{\{u_{*,j}>0\}} - f_{j}\chi_{\cap_{\ell}\{u_{*,\ell}=0\}} \quad \text{in $\mR^{n}$.} 
\end{equation*}
Note that this, together with \eqref{eq:boundedness}, implies that 
$-\Delta\left(\sum_{\ell\neq j}u_{*,\ell}-u_{*,j}\right) \le F$, where $F \in L^{\infty}(\mathbb{R}^n)$ and has compact support. Hence, $\sum_{\ell\neq j}u_{*,\ell}-u_{*,j}-U^F$ is subharmonic, which in turn implies that $\sum_{\ell\neq j}u_{*,\ell}-u_{*,j}=U^F-U^{\eta}$ for some $\eta \ge 0$.
It follows, in particular, that the left-hand side in \eqref{eq:Euler-Lagrange-general} is a measure, and it is moreover necessarily singular with respect to the Lebesgue measure on the set $\bigcap_{\ell}\{u_{*,\ell}=0\}$ (since the harmonic measure for any open set satisfies this). Hence the estimate \eqref{eq:Euler-Lagrange-general} follows. 
\end{proof}

\begin{corollary}\label{cor:cont-min-Sm}
    Let $\bff=(f_{1},\cdots,f_{m})$ where each function $f_{1},\cdots,f_{m} \in L^{\infty}(\mR^{n})$ satisfies \eqref{eq:assumption-A}. If $\bfu_{*}=(u_{*,1},\cdots,u_{*,m})$ is a minimizer of $\mJ_{\bff}$ over $\mS_{m}$, then it has a continuous representative, i.e. such that $u_{*,i} \in C_c(\mR^n)$ for each $i=1,2,\cdots,m$.
\end{corollary}

\begin{proof}
We note that \Cref{lem:improvement-prop1-AS16} above allows us to first refine the choice of representatives from \Cref{rem:quasicont-rep} for the minimizers. Because by \eqref{eq:boundedness} and \eqref{eq:monotonicity-support}, the minimizer $\mathbf{u}_{*}$ has compact support. Consequently, the right-hand side of \eqref{eq:Euler-Lagrange-general}, denoted as $F_j$, belongs to $L^{\infty}(\mR^{n})$ and has compact support. It follows that its Newtonian potential $U^{F_j}$ lies in $C_{\rm loc}^{1}(\mR^{n})$, and it now follows from \eqref{eq:Euler-Lagrange-general} that
\begin{equation*}
\eta_j:=\Delta \left(\sum_{\ell\neq j}u_{*,\ell}-u_{*,j}-U^{F_j}\right) \ge 0 \quad \text{in $\mR^{n}$,}
\end{equation*}
where $\eta_j$ is a Radon measure with compact support.
Therefore (see, e.g., \cite[Theorem~4.1.8]{Hoermander_book_1})
\[\sum_{\ell\neq j}u_{*,\ell}-u_{*,j}-U^{F_j}=-U^{\eta_j} \textrm{ (a.e.) }\] 
where, in particular, the right hand side is USC. Hence we have that
\[u_{*,j}-\sum_{\ell\neq j}u_{*,\ell} = U^{\eta_j} - U^{F_j} \textrm{ (a.e.)},\]
and the right hand side is LSC. So the set $Q_j=\{U^{\eta_j} - U^{F_j}>0\}$ is open. 
It is now clear that $u_{*,j}=w_j$ (a.e.), where
\[w_j={\rm max}\{U^{\eta_j} - U^{F_j},0\} = \begin{cases}
    U^{\eta_j} - U^{F_j} & \textrm{ in } Q_j,\\ 0 & \textrm{ in } Q_j^c.
\end{cases}\]
If we use these representatives $(w_1,w_2,\cdots,w_m)$, then we note that inside $Q_{\ell}$ we have by \eqref{eq:Euler-Lagrange-general} that $-\Delta w_{\ell}\le f_{\ell}$ for $\ell \ne j$, and $-\Delta w_j\ge f_j$ in $Q_j$. But interchanging the role of $j$ and some other $\ell$ then shows that we must actually have equality here. It is also clear that the sets $Q_1,Q_2,\cdots,Q_m$ are pairwise disjoint.
It follows by Kato's inequality (see, e.g., \cite{BP04KatoInequality} or \cite[Corollary~2.3]{GS12TwoPhaseQD} for a short alternative proof) that
\[\gamma_i= \Delta w_i|_{Q_i^c} \ge 0,\]
where $\gamma_i$ is supported on $\partial Q_i$. And furthermore, since $F_j=-f_j$ a.e. in $Q_j$,
\[w_i=U^{f_i|_{Q_i}}-U^{\gamma_i} \textrm{ (q.e.)},\]
and the only points where this equality potentially fails are those $y \in \partial Q_i$ such that $\mR^n \setminus Q_i$ is thin at $y$ (i.e. irregular for Dirichlet's problem). 
All involved functions, due to \eqref{eq:monotonicity-support}, are bounded, and we now have the equality 
\begin{equation*}
\begin{aligned}
w_j-\sum_{\ell\neq j}w_{\ell} &= U^{\eta_j} - U^{F_j} = U^{f_j|_{Q_j}}-U^{\gamma_j} +\sum_{\ell \ne j} U^{\gamma_{\ell}} - \sum_{\ell \ne j} U^{f_{\ell}|_{Q_{\ell}}} \\
&= \sum_{\ell \ne j} U^{\gamma_{\ell}} - U^{\gamma_j} - U^{F_j}  \textrm{ (q.e.)},
\end{aligned}
\end{equation*}
since $F_j=\sum_{\ell \ne j} f_{\ell}|_{Q_{\ell}} - f_j|_{Q_j}$ (a.e.).
From this we see that we must have $\eta_j=\sum_{\ell \ne j} \gamma_l - \gamma_j$.
It follows that the only way that $w_j$ could be discontinuous is that we have a point $y$ on $\partial Q_j$ where $\mR^n \setminus Q_j$ is thin, and such that 
\begin{equation*}
(U^{f_j|_{Q_j}}-U^{\gamma_j})(y) = \limsup_{x \to y, x \in Q_j} (U^{f_j|_{Q_j}}-U^{\gamma_j})(x)>0. 
\end{equation*}
But if this is the case, then for all $\ell \ne j$ the complement of $Q_{\ell}$ is non-thin at $y$, and therefore $U^{f_{\ell}|_{Q_{\ell}}} - U^{\gamma_{\ell}}$ is continuous at this point with value $0$. Hence $(U^{\eta_j} - U^{F_j})(y)=(U^{f_j|_{Q_j}}-U^{\gamma_j})(y)>0$.
But then we have by LSC
\[0<(U^{\eta_j} - U^{F_j})(y) \le \liminf_{x \to y} (U^{\eta_j} - U^{F_j})(x) \le 0\]
which gives a contradiction. Hence the functions $U^{f_i|_{Q_i}}-U^{\gamma_i}$ are continuous for all $i$, so that the minimizer always has a continuous representative.
\end{proof}

The following is now a more or less immediate consequence of \Cref{lem:improvement-prop1-AS16} above:
\begin{corollary}\label{cor:sufficient-mQD}
Let $\mu_{1},\cdots,\mu_{m}\in L^{\infty}(\mR^{n})$ be measures with compact support and let $\bff$ satisfy \eqref{eq:special-choice-f}. Suppose that $\bfu_{*}=(u_{*,1},\cdots,u_{*,m})$ is a minimizer of $\mJ_{\bff}$ over $\mS_{m}$ that satisfies the support condition 
\begin{equation*}
\supp\,(\mu_{i})\subset Q_{i}:=\{u_{*,i}>0\} \quad \text{for all $i=1,\cdots,m$}, 
\end{equation*}
then $(Q_{1},\cdots,Q_{m})$ is a $m$-phase quadrature domain with respect to $(\mu_{1},\cdots,\mu_{m})$ in the sense of \Cref{def:m-phase-QD}. 
\end{corollary} 

\begin{lemma}\label{lem:support-of-measure-minimizer}
Let $\bfu_{*}=(u_{*,1},\cdots,u_{*,m})$ denote a minimizer of $\mJ_{\bff}$ over $\mS_{m}$, where $\bff$ satisfies \eqref{eq:special-choice-f}. Then $\mu_j \leq 1$ in $\mR^n \setminus \left(\overline{ \bigcup_{\ell\ne j}  \{u_{*,\ell}>0\}} \cup \{u_{*,j}>0\}\right)$ for every $j$.
\end{lemma}

\begin{proof}
Suppose this is not the case. Then we may without loss of generality assume that $j=1$. Then there is some $c>1$ and some suitable small compact subset $K$ of $\mR^n \setminus \left(\overline{ \bigcup_{\ell\ne 1}  \{u_{*,\ell}>0\}} \cup \{u_{*,1}>0\}\right)$ of positive measure such that $\mu_1 >c\chi_{K}$ and also $\omega(c\chi_K) \subset\mR^n \setminus \overline{ \bigcup_{\ell \ne 1} \{u_{*,1}>0\}}$. 
We may now add $u=W^{c\sfm|_K}$ to $u_1$, and then it is immediate by construction that $(u_{*,1}+u,u_{*,2},\cdots,u_{*,m})$ still belongs to $\mS_m$. 
Furthermore, since we have $(c\chi_K-1)-2f_1\chi_{\{u_{*,1}=0\}} \le 1-c\chi_K$, $\nabla u_{*,1}=0$ a.e. in $\{ u_{*,1}=0\} \setminus \overline{ \bigcup_{\ell \ne 1} \{u_{*,1}>0\}}$ and $-\Delta u_{*,1}=f_1\chi_{\{u_{*,1}>0\}}$ in $\{u>0\}$, we get
\begin{align*}
    J_{f_1}(u_{*,1}+u) &= \int|\nabla u_{*,1}+\nabla u|^2 \,\rmd\sfm - 2 \int f_1(u_{*,1}+u) \,\rmd\sfm \\
    &= J_{f_1}(u_{*,1})+ \int |\nabla u|^2 \,\rmd\sfm + 2\int \nabla u \cdot \nabla u_{*,1}d\sfm - 2\int f_1u \,\rmd\sfm\\
    &=J_{f_1}(u_{*,1})+\int u(c\chi_K-1) \,\rmd\sfm +2\int u f_1 \chi_{\{u_{*,1}>0\}} \,\rmd\sfm - 2\int f_1u \,\rmd\sfm \\
    &\le J_{f_1}(u_{*,1})+ \int u(c\chi_K-1) \,\rmd\sfm - 2\int f_1 u \chi_{\{u_{*,1}=0\}} \,\rmd\sfm \\
    & \le J_{f_1}(u_{*,1}) - \int u(c\chi_K-1) \,\rmd\sfm = J_{f_1}(u_{*,1})-\int |\nabla u|^2 \,\rmd\sfm <J_{f_1}(u_{*,1}), 
\end{align*}
which contradicts that $\bfu_{*}$ is a minimizer.
\end{proof}

Using \Cref{lem:improvement-prop1-AS16}, we deduce that if $\bfu=(u_{1},\cdots,u_{m})\in\mS_{m}$ is a minimizer of $\mJ_{\bff}$ over $\mS_{m}$ with $f_{i}=\mu_{i}-1$ for each $i=1,\cdots,m$, then it satisfies
\begin{equation}\label{eq:minsmmu-cond-1}
-\Delta\left(\sum_{\ell \ne j} u_{\ell}-u_{j}\right) \leq \sum_{\ell \ne j}(\mu_{\ell}-1)\chi_{\{u_{\ell}>0\}}-(\mu_{j}-1)\chi_{\{u_{j}>0\}} \quad\text{in $\mR^{n}$.}
\end{equation}
However, the converse does not hold in general: a function in $\mS_{m}$ satisfying \eqref{eq:minsmmu-cond-1} need not be a minimizer of $\mJ_{\bff}$ over $\mS_{m}$. We now show that this partial converse becomes valid when restricted to the smaller class $\mS_{m,\bmmu}$ defined by 

\begin{equation}
\mS_{m,\bmmu} = \left\{ \bfu=(u_{1},\cdots,u_{m})\in\mS_{m} : \mu_{i}(\{u_{i}=0\})=0 \text{ for all $i$} \right\}. \label{eq:smaller-class}
\end{equation}
Here, we do not assume that $\mu_{i}$ has compact support in $\{u_{i}>0\}$, we only require that it does not charge the complement, that is, $\mu_{i}(\{u_{i}=0\})=0$. 
It is worth noting, however, that $\mS_{m,\bmmu}$ is not closed. Consequently, a minimizer over this class need not exist in general. 

\begin{proposition}\label{prop:uniqueness-procedure}
Let $\bfu_{*}=(u_{*,1},\cdots,u_{*,m})\in\mS_{m,\bmmu}$ and $\bmmu=(\mu_{1},\cdots,\mu_{m})$ for each $j=1,\cdots,m$ satisfy \eqref{eq:minsmmu-cond-1}. 
Then $\bfu_{*}$ is the \emph{unique} minimizer of $\mJ_{\bff}$ over $\mS_{m,\bmmu}$, where $f_i=\mu_i-1$ for each $i=1,2,\cdots,m$.  
\end{proposition}

In particular, in view of \Cref{lem:average-L-infty}, \Cref{thm:main-3} follows essentially as a direct consequence of \Cref{prop:uniqueness-procedure}.

\begin{proof}[Proof of \Cref{prop:uniqueness-procedure}]
Suppose $\bfv \in \mS_{m,\bmmu}$ is another element, and write its components as $v_i=u_{*,i}+g_i$. Note then the following properties to be used below that $g_i$ satisfies by assumption
\begin{itemize}
\item $g_i=v_i$ when $u_{*,i}=0$, so in particular $g_i \geq 0$ if $u_{*,i}=0$.
\item If $\ell \ne i$ and $u_{*,i}(x)+g_i(x)>0$, then $u_{*,\ell}(x)=-g_\ell(x)$.
\item If $\ell \ne i$ then $g_if_\ell=-g_i$, because $v_i=g_i=0$ on the support of $\mu_\ell$.  
\end{itemize}
Now we may compute
\begin{align*}
&J_{f_i}(u_{*,i}+g_i)=\int |\nabla u_{*,i} + \nabla g_i|^2\,\rmd \sfm-2\int (u_{*,i}+g_i)f_i\,\rmd \sfm\\
&= \int |\nabla u_{*,i}|^2\,\rmd \sfm-2\int u_{*,i}f_i\,\rmd \sfm +\int |\nabla g_i|^2\,\rmd \sfm
+2\int \nabla u_{*,i} \cdot \nabla g_i\,\rmd \sfm-2\int g_if_i\,\rmd \sfm\\
&= J_{f_i}(u_{*,i})+\int |\nabla g_i|^2\,\rmd \sfm
+2\int \nabla u_{*,i} \cdot \nabla g_i^+\,\rmd \sfm-2\int g_i^+f_i\,\rmd \sfm \\
&\quad -2\int \nabla u_{*,i} \cdot \nabla g_i^-\,\rmd \sfm+2\int g_i^-f_i\,\rmd \sfm.
\end{align*}
Since $g_i^-$ is zero when $u_{*,i}=0$ and $\Delta u_{*,i}=f_i$ when $u_{*,i}>0$ the last two terms above cancel each other. Hence
\begin{equation}
J_{f_i}(u_{*,i}+g_i)= J_{f_i}(u_{*,i})+\int |\nabla g_i|^2\,\rmd \sfm
+2\int \nabla u_{*,i} \cdot \nabla g_i^+\,\rmd \sfm-2\int g_i^+f_i\,\rmd \sfm, \label{eq:computation1}
\end{equation}
By assumption we have for each test function $\psi \geq 0$
\begin{equation}\label{eq:computation3}
    \sum_{\ell \ne i} \int \nabla u_{*,\ell} \cdot \nabla \psi \,\rmd \sfm - \int \nabla u_{*,i} \cdot \nabla \psi \,\rmd \sfm \leq \sum_{\ell \ne i} \int_{\{u_{*,\ell}>0\}} f_\ell \psi \,\rmd \sfm-\int_{\{u_{*,i}>0\}} f_i\psi \,\rmd \sfm.
\end{equation}
By approximation the above holds for any $\psi \geq 0$ in $H^1(\mR^n)$, hence we get by rearranging the terms above and letting $\psi = g_i^+$
\begin{align*}
    \int \nabla u_{*,i} \cdot \nabla g_i^+\,\rmd \sfm &\geq \sum_{\ell \ne i} \int \nabla u_{*,\ell} \cdot \nabla g_i^+ \,\rmd \sfm -\sum_{\ell \ne i} \int_{\{u_{*,\ell}>0\}} f_\ell g_i^+ \,\rmd \sfm + \int_{\{u_{*,i}>0\}} f_ig_i^+ \,\rmd \sfm\\
    & = - \sum_{\ell \ne i} \int_{\{u_{*,\ell}>0\}} \nabla g_\ell \cdot \nabla g_i \,\rmd \sfm -\sum_{\ell \ne i} \int_{\{u_{*,\ell}>0\}} f_\ell g_i \,\rmd \sfm + \int_{\{u_{*,i}>0\}} f_i g_i^+ \,\rmd \sfm\\
    &\geq - \sum_{\ell \ne i} \int_{\{u_{*,\ell}>0\}} \nabla g_\ell \cdot \nabla g_i \,\rmd \sfm  + \int_{\{u_{*,i}>0\}} f_i g_i^+ \,\rmd \sfm,
\end{align*}
where we used that $g_i^+=g_i$ when $u_{*,i}=0$, $u_{*,\ell}=-g_\ell$ when both $g_i \ne 0$ and $u_{*,\ell}>0$ and finally that $f_\ell g_i=-g_i \le 0$ whenever $u_{*,\ell}>0$ and $\ell \ne i$.

If we insert this into (\ref{eq:computation1}) and sum over $i$ we get
\begin{equation}\label{eq:computation2}
    \mJ_{\bff}(\bfv) \geq \mJ_{\bff}(\bfu_{*})+ \sum_{i=1}^{m} \left(\int |\nabla g_i|^2\,\rmd \sfm-2\sum_{\ell \ne i} \int_{\{u_{*,\ell}>0\}} \nabla g_i \cdot \nabla g_\ell \,\rmd \sfm\right),
\end{equation}
where we used that $\int_{\{u_{*,i}=0\}} g_i^+f_i d\sfm <0$.

We now note that for any given $j,k$ with $j \ne k$ if $u_{*,j}(x)>0$ and $g_k(x)>0$, then $g_\ell(x)=0$ for all $\ell \ne j,k$ by assumption. Since we also have $\nabla g_{\ell}(x)=0$ for a.e. $x$ such that $g_{\ell}(x)=0$ we therefore get for a.e. such $x$
\begin{align*}
&\sum_{i=1}^{m}|\nabla g_i(x)|^2\chi_{\{u_{*,i}>0\}}(x)-2\sum_{i=1}^{m}\sum_{\ell \ne i} \nabla g_i(x) \cdot \nabla g_\ell(x)\chi_{\{u_{*,\ell}>0\}}(x) \\
& \quad = |\nabla g_j(x)|^2+|\nabla g_k(x)|^2-2\nabla g_k(x) \cdot \nabla g_j(x) \geq 0.
\end{align*}
From this we see that indeed 
\[\sum_{i=1}^{m} \left(\int |\nabla g_i|^2\,\rmd \sfm-2\sum_{\ell \ne i} \int_{\{u_{*,\ell}>0\}} \nabla g_i \cdot \nabla g_\ell \,\rmd \sfm\right) \geq \sum_{i=1}^{m} \int_{\{u_{*,i}=0\}} |\nabla g_i|^2\,\rmd \sfm \geq  0.\]
This together with (\ref{eq:computation2}) above proves that $\bfu_{*}$ is a minimizer.

When it comes to the uniqueness part, note that in case there is a part where $u_{*,i}=0$ but $g_i>0$, then we get a strict inequality in the above estimates since $\int_{\{u_{*,i}=0\}} |\nabla g_i|^2 \rmd \sfm >0$ then (because $\{g_i>0\} \setminus \overline{\{u_{*,i}>0\}}$ can't be empty in this situation, and $g_i$ has compact support, so we can't have $\nabla g_i=0$ a.e. in $\mR^n \setminus \overline{\{u_{*,i}>0\}}$). This proves that $\{v_i>0\} \subset \{u_{*,i}>0\}$ in case $\bfv$ is also a minimizer in $\mS_{m,\bmmu}$. 
But then we would get, since $g_i=0$ when $u_{*,i}=0$, that  
\begin{align*}
&J_{f_i}(u_{*,i}+g_i)= J_{f_i}(u_{*,i}) +\int |\nabla g_i|^2\,\rmd \sfm +2\int_{\{u_{*,i}>0\}} \nabla u_{*,i} \cdot \nabla g_i\,\rmd \sfm-\int g_i f_i\,\rmd \sfm\\
&=J_{f_i}(u_{*,i})+\frac{1}{2}\int |\nabla g_i|^2\,\rmd \sfm.
\end{align*}
And unless $\int |\nabla g_i|^2\,\rmd \sfm=0$ for each $i$ we get a contradiction to that $\bfv$ is a minimizer, which proves that $g_i=0$ for all $i$. 
\end{proof}

We conclude this section with a few examples illustrating natural limitations of the energy minimization approach, already visible in the two-phase setting. 
First, we present an example showing that there exist two-phase quadrature domains in the plane $\mR^{2}$ that do not minimize $\mJ_{\bff}$ over $\mS_2$ (although they do minimize it over $\mS_{2,\bmmu}$, by \Cref{prop:uniqueness-procedure}). 
In other words, unlike the two-phase partial balayage approach in \cite{GS12TwoPhaseQD}, the energy minimization framework described above does not necessarily capture all two-phase quadrature domains. 
After that we provide an example showing that minimizers over $\mS_2$ are not necessarily unique.

\begin{example}\label{ex:counterexample-tpqd-v3}
Let $\mu_1,\mu_2\in L^{\infty}(\mR^{n})$ be nonnegative and satisfy the assumptions of \cite[Theorem~5.1]{GS12TwoPhaseQD}. Define $u_1=(\overline{W}^{\mu_{1}-\mu_{2}})^+$ and $u_2=(\overline{W}^{\mu_{1}-\mu_{2}})^-$ as in \eqref{eq:def-2-phase-balayage}. Then
\begin{equation*}
(Q_1,Q_2) := \left(\Omega^{+}(\mu_{1}-\mu_{2}),\Omega^{-}(\mu_{1}-\mu_{2})\right) = \left(\{u_1>0\},\{u_2>0\}\right) 
\end{equation*}
is a two-phase quadrature domain with respect to $(\mu_1,\mu_2)$. Assume further that $\omega(\mu_1) \cap \omega(\mu_2) \ne \emptyset$, then, in particular, we have the strict inequality 
\[J_{f_1}(u_1)>J_{f_1}(u),\]
where $u=W^{\mu_{1}}$ and $f_1=\mu_1-1\in L^{\infty}(\mR^{n})$. 

We now show that, if necessary, we can modify the function $u_2$ and the measure $\mu_2$ to obtain $\tilde{u}_2$ and $\tilde{\mu}_2$ such that $(u_1,\tilde{u}_2)$ satisfies the definition of a two-phase quadrature domain with respect to $(\mu_1,\tilde{\mu}_2)$, but 
\begin{equation}
\mJ_{(f_{1},\tilde{f}_{2})}(u,\tilde{u}_{2}) = J_{f_1}(u_1)+J_{\tilde{f}_2}(\tilde{u}_2)>J_{f_1}(u)=\mJ_{(f_{1},\tilde{f}_{2})}(u,0), \label{eq:to-do-strict-energy}
\end{equation}
where $\tilde{f}_{2}=\tilde{\mu}_{2}-1\in L^{\infty}(\mR^{n})$. This demonstrates that $(u_1,\tilde{u}_2)$, despite defining a valid two-phase quadrature domain, does not minimize the corresponding functional over $\mS_2$. 

To achieve this, choose $\epsilon>0$ small enough so that $\supp(\mu_2) \subset \{u_2 > \epsilon\}$, and define 
\begin{equation*}
w_2=\min\{u_2,\varepsilon\}. 
\end{equation*}
Next, set 
\[\eta_2=-\Delta w_2+\sfm|_{Q_1} =\gamma + \sfm|_{\{u_2>\epsilon\}},\]
where $\gamma$ is supported on $\{u_2=\epsilon\}$, a compact subset of $\{u_2>0\}$. The total mass satisfies $\gamma(Q_2)=\mu_2(Q_2)-\sfm(\{u_2>\epsilon\})$. 
It is straightforward to verify that $(u_1,w_2)$ satisfies the definition of a strong two-phase quadrature domain for subharmonic functions with respect to $(\mu_1,\eta_2)$. 
Moreover, the energy satisfies 
\begin{equation*}
\int \abs{\nabla w_2}^2 \,\rmd\sfm -2 \int w_2 \,\rmd\gamma + 2\int w_2 \,\rmd\sfm \to 0 \quad \text{as $\epsilon\rightarrow 0$.} 
\end{equation*}
Although $\gamma\notin L^{\infty}$ (it does have finite energy), this can be handled as in \Cref{lem:average-L-infty}: choose $\tau>0$ with $\tau < \epsilon/2$, replace $\gamma$ by the mollified measure $\tilde{\mu}_2=\gamma^{\tau}=\gamma*\frac{1}{\abs{B_{\tau}}}\chi_{B_{\tau}}$ and define
\begin{equation*}
\tilde u_2=w_2+U^{\gamma^{\tau}}-U^{\gamma}. 
\end{equation*} 
For suitably small $\tau,\epsilon$ we then have \eqref{eq:to-do-strict-energy}, showing that $(u_1,\tilde{u}_2)$ is not a minimizer of $\mJ_{(f_1,\tilde{f}_2)}$ over $\mS_2$, despite defining a valid two-phase quadrature domain. 
\end{example}

\begin{example} \label{ex:counterexample-unique-min}
Here we aim to illustrate that minimizers over $\mS_m$ are, in general, not unique. This phenomenon occurs already for $m=2$. The example below also demonstrates the non-existence of certain multiphase quadrature domains, highlighting the delicate nature of existence questions. To simplify the discussion, and in view of \Cref{lem:average-L-infty}, we allow measures with finite energy that need not belong to $L^{\infty}$ (a standard mollification argument can, as before, reduce to the $L^{\infty}$-case), and we restrict our analysis to $\mR^2$. 

For given radii $0<R_1<R_2$, consider the function 
\begin{equation*}
W(R_1,R_2,r)= \left\{\begin{aligned}
& \frac{r^2}{4} + C_1, && 0 \le r < R_1,\\
& \frac{r^2}{4} + C_2 \ln r + C_3, && R_1 \le r <R_2,\\
& 0, && r\ge R_2,
\end{aligned}\right. 
\end{equation*}
where 
\begin{equation*} 
C_1= \frac{R_2^2}{2}\ln(R_2/R_1)-\frac{R_2^2}{4}, \quad C_2=-\frac{R_2^2}{2}, \quad C_3=\frac{R_2^2}{2}\ln R_2 -\frac{R_2^2}{4}
\end{equation*} 
are chosen so that $W(R_1,R_2,R_2)=W'(R_1,R_2,R_2)=0$. Then we have
\begin{equation*} 
W^{\mu_{(R_1,R_2)}}(x)=W(R_1,R_2,|x|)
\end{equation*} 
where
\begin{equation*} 
\mu_{(R_1,R_2)}= \frac{R_2^2}{2R_1}\mathcal{H}^{1}|_{\partial B_{R_1}}.
\end{equation*} 
We set $f_{(R_1,R_2)}=\mu_{(R_1,R_2)}-1$, so that the corresponding energy functional is given by 
\[J_{f_{(R_1,R_2)}}(u) = \int |\nabla u|^2 \, \rmd\sfm -2\int u \, \rmd\mu_{(R_1,R_2)} + 2\int u \, \rmd\sfm.\]
Choosing $u(x)=W(R_1,R_2,|x|)$, the energy reduces to 
\begin{equation*}
J_{f_{(R_1,R_2)}}(W(R_1,R_2,|x|)) = -\int |\nabla W(R_1,R_2,|x|)|^2 \, \rmd\sfm = -2\pi \int_{0}^{R_2} r(W'(R_1,R_2,r))^2\,\rmd r,
\end{equation*}
and a straightforward computation yields 
\begin{equation*} 
J_{f_{(R_1,R_2)}}(W(R_1,R_2,|x|)) = \frac{\pi}{8}\left(3R_2^4-4R_2^4\ln(R_2/R_1)-4R_1^2R_2^2\right).
\end{equation*} 

There exists $R\in(5,5.1)$ such that the energy of $W(4,16,|x|)$ equals that of $W(R,17,|x|)$. Let $\mu_1=\mu_{(4,16)}$ and $\mu_2=\mu_{(R,17)}$ for this value of $R$. 
We now show, admittedly somewhat tediously, that there is no two-phase quadrature domain corresponding to $(\mu_1,\mu_2)$. To see this, note that the only candidates for a two-phase quadrature domain $(v_1,v_2)$ have the following form (with $r=\abs{x}$): 
\begin{equation*} 
v_1(r)=\left\{\begin{aligned}
&0 && r <r_1 && (\text{omitted if $r_1=0$}),\\
&\frac{r^2}{4}+d_1\ln r + d_2 && r_1 \le r < 4 && (\text{the term $d_1\ln r$ is omitted if $r_1=0$}), \\
&\frac{r^2}{4}+d_3 \ln r + d_4 && 4 \le r < r_2,\\
& 0 && r \ge r_2
\end{aligned}\right. 
\end{equation*} 
and 
\begin{equation*}
v_2(r)=\left\{\begin{aligned}
& 0 && r <r_2,\\
& \frac{r^2}{4}+d_5\ln r + d_6 && r_2 \le r < R,\\
& \frac{r^2}{4}+d_7 \ln r + d_8 && R \le r < r_3,\\
& 0 && r \ge r_3,
\end{aligned}\right. 
\end{equation*}
where $0 \le r_1 <4 < r_2 < R < r_3 \le 17$ (the upper bound $r_3 \le 17$ follows since a two-phase quadrature domain must lie within the corresponding one-phase quadrature domains), and $d_1,d_2,\cdots,d_8$ are constants. 

In the case $r_1>0$, we obtain $11$ equations that must be satisfied: 
\begin{subequations} 
\begin{align}
& \frac{r_1^2}{4}+d_1\ln r_1 +d_2=0 &&  \because\text{\footnotesize continuity of $v_1$ in $r_1$ (if $r_1>0$)},\label{eq:1} \\ 
&  \frac{r_1}{2}+\frac{d_1}{r_1}=0 && \because\text{\footnotesize $(v_1)_+'(r_1)=0$ (if $r_1>0$)},\label{eq:2} \\
&  d_3\ln 4+d_4-d_1\ln 4-d_2=0 && \because\text{\footnotesize continuity of $v_1$ at $r=4$},\label{eq:3} \\
& \frac{r_2^2}{4}+d_3\ln r_2 +d_4=0 && \because\text{\footnotesize continuity of $v_1$ at $r=r_2$},\label{eq:4}\\
&  \frac{r_2^2}{4}+d_5\ln r_2 +d_6=0 && \because\text{\footnotesize continuity of $v_2$ at $r=r_2$},\label{eq:5}\\
&  d_7\ln R+d_8-d_5\ln R -d_6=0 && \because\text{\footnotesize continuity of $v_2$ at $r=R$},\\
&\frac{r_3^2}{4}+d_7\ln r_3 +d_8=0 && \because\text{\footnotesize continuity of $v_2$ at $r=r_3$},\\
& \frac{r_3}{2}+\frac{d_7}{r_3}=0 && \because\text{\footnotesize $(v_2)_-'(r_3)=0$}, \label{eq:8}\\
& d_3-d_1=-16^2/2 && \because\text{\footnotesize jump condition at $r=4$ to get $\mu_1$}, \label{eq:9}\\
& d_3+d_5=-r_2^2 && \because\text{\footnotesize $(v_2)'_+(r_2)+(v_1)'_-(r_2)=0$},\\
& d_7-d_5= -17^2/2 && \because\text{\footnotesize jump condition at $r=R$ to get $\mu_2$}. \label{eq:11}
\end{align}
\end{subequations} 

Restricting attention to equations \eqref{eq:1}--\eqref{eq:4} for given $r_1,r_2,r_3$, we obtain unique values for $d_1,d_2,d_3,d_4$:
\begin{equation*} 
d_1 = -\frac{r_1^2}{2}, \quad d_2 = \frac{r_1^2}{2}\ln r_1-\frac{r_1^2}{4}, \quad d_3 = \frac{2r_1^2\ln(4/r_1)+r_1^2-r_2^2}{4\ln(r_2/4)}, \quad d_4 = - \frac{r_2^{2}}{4} -d_3\ln r_2. 
\end{equation*} 
Similarly, restricting attention to equations \eqref{eq:5}--\eqref{eq:8} for given $r_1,r_2,r_3$, we obtain unique values for $d_5,d_6,d_7,d_8$:
\begin{equation*}
d_7 = -\frac{r_3^2}{2}, \quad d_8 = \frac{r_3^2}{2}\ln r_3 - \frac{r_3^2}{4}, \quad d_5 = \frac{2\ln(r_3/R)+r_{2}^{2}-r_{3}^{2}}{4\ln(R/r_2)}, \quad d_6 = -\frac{r_2^2}{4}-d_5\ln r_2. 
\end{equation*} 
Substituting $d_1$ and $d_3$ into \eqref{eq:9} produces 
\begin{equation*}
2r_1^2\ln(r_2/r_1)+r_1^2-r_2^2=-512\ln(r_2/4). 
\end{equation*}
Since $r_1<4<r_2$, we in particular have $r_2^2-512\ln(r_2/4)>0$, and a simple analysis shows that this forces $4<r_2<4.2$.
Moreover, because $(v_1,v_2)$ corresponds to a two-phase quadrature domain by assumption, this yields the following matching of areas with respect to the measures $\mu_1$ and $\mu_2$: 
\begin{equation*} 
\pi(r_3^2-r_2^2) - \pi(r_2^2-r_1^2)= \|\mu_2\|-\|\mu_1\| =\pi \cdot 17^2-\pi \cdot 16^2=33\pi.
\end{equation*} 
Since $r_2<4.2$, we obtain 
\begin{equation*} 
r_3^2 \le 33 + 2 \cdot r_2^2 < 33+2 \cdot 4.2^2 =68.28,
\end{equation*} 
which gives $r_3<8.3.$
Substituting $d_5$ and $d_7$ into \eqref{eq:11} produces 
\begin{equation*} 
-2r_3^2\ln(r_3/R)-2r_3^2\ln(R/r_2)+r_3^2-r_2^2 = -17^2 \cdot 2 \ln(R/r_2)
\end{equation*} 
which is equivalently written as 
\begin{equation*} 
-2r_3^2 \ln(r_3/r_2)+r_3^2-r_2^2+578\ln(R/r_2) =0.
\end{equation*} 
However, using estimates $r_2 < 4.2$, $5<r_3 < 8.3$ and $R>5$, the left-hand side admits the lower bound 
\begin{equation*} 
-2 \cdot 8.3^2 \ln(8.3/4)+5^2-4.2^2+578 \ln(5/4.2)\approx 7.6>0
\end{equation*} 
so there is no solution to the system. 

Finally, consider the case $r_1=0$. Then $d_1=0$, so equations \eqref{eq:1} and \eqref{eq:2} are omitted, while equations \eqref{eq:3}--\eqref{eq:11} remain the same with $d_1=0$ throughout. We need a slightly different argument here. First, the area comparison 
\begin{equation*} 
\pi(r_{3}^{2}-2r_{2}^{2})= \pi(r_3^2-r_2^2) - \pi r_2^2= \|\mu_2\|-\|\mu_1\| =\pi \cdot 17^2-\pi \cdot 16^2=33\pi 
\end{equation*} 
now yields  
\begin{equation*} 
r_3^2=33+2r_2^2.
\end{equation*} 
Hence, equation \eqref{eq:11} now gives 
\begin{align*}
d_7-d_5&=-\frac{33+2r_2^2}{2} - \frac{r_2^2-33-2r_2^2+2\ln(\sqrt{33+2r_2^2}/R_1)}{4\ln(R/r_2)}= -\frac{17^2}{2}\\
\iff &
\frac{33+r_2^2-\ln((33+2r_2^2)/R^2)}{4\ln(R/r_2)}=\frac{33+2r_2^2-17^2}{2}\\
\iff &  33+r_2^2-\ln((33+2r_2^2)/R^2) + 2(33+2r_2^2-17^2)\ln(r_2/R)=0\\
\iff & \ln(r_2/R) = \frac{33+r_2^2-\ln((33+2r_2^2)/R^2)}{512-4r_2^2}.
\end{align*}
However, this implies, since $4<r_2<R$ and $5<R<5.1$, that 
\begin{equation*} 
\ln(r_2/R) \ge \frac{33+4^2-\ln((33+2 \cdot R^2)/R^2)}{512-4 \cdot 4^2} \ge \frac{33+4^2-\ln(2 +33/5.1^2)}{512-4 \cdot 4^2} \ge 0.1 
\end{equation*} 
which in turn gives
\begin{equation*} 
r_2/R \ge {\rm e}^{0.1} \ge 1.1,
\end{equation*} 
contradicting $r_2 \le R$. Hence, no solution exists under these constraints. 

Consequently, in light of \Cref{lem:improvement-prop1-AS16},  the only candidates for minimizers $(u_1,u_2) \in \mS_2$ of the energy functional 
\begin{equation*} 
\mJ_{\bff}(u_1,u_2)=J_{f_1}(u_1)+J_{f_2}(u_2)
\end{equation*}  
are $(W(4,16,|x|),0)$ and $(0,W(R_1,17,|x|)$. By construction, both candidates attain the same energy, so the minimizer in this case is not unique.  
\end{example}

\section{Proof of \texorpdfstring{\Cref{thm:main-1-rough}}{Theorem \ref{thm:main-1-rough}} \label{sec:result1-proof-ts}}
Fix distinct indices $i$ and $j$.
By \Cref{lem:improvement-prop1-AS16}, the distribution $\Delta\left(\sum_{\ell\neq j}u_{*,\ell}-u_{*,j}\right)$ is a signed Radon measure in $\mR^{n}$, and  
\begin{equation*}
-\Delta(u_{*,i}-u_{*,j}) \le f_{i}\chi_{\{u_{*,i}>0\}} - f_{j}\chi_{\{u_{*,j}>0\}} \quad \text{in $\mR^{n}\setminus\bigcup_{\ell\neq i,j}\overline{\{u_{*,\ell}>0\}}$.} 
\end{equation*}
Interchanging the roles of $i$ and $j$ yields
\begin{equation*}
-\Delta(u_{*,i}-u_{*,j}) = f_{i}\chi_{\{u_{*,i}>0\}} - f_{j}\chi_{\{u_{*,j}>0\}} \quad \text{in $\mR^{n}\setminus\bigcup_{\ell\neq i,j}\overline{\{u_{*,\ell}>0\}}$.} 
\end{equation*}
In particular 
\begin{equation}
-\Delta u_{*,i} = f_{i}\chi_{\{u_{*,i}>0\}} \quad \text{in $\mR^{n}\setminus\bigcup_{\ell\neq i}\overline{\{u_{*,\ell}>0\}}$.} \label{eq:m-phase-PDE-TS}
\end{equation}

\begin{theorem}\label{thm:main-1-ts1}
    Suppose $\mu_1,\mu_2,\ldots,\mu_m$ are such that 
    \begin{equation*}
        \supp(\mu_i) \cap \overline{\omega(\mu_j)}= \emptyset \textrm{ for all } i \ne j 
    \end{equation*}
    and
    \begin{equation*}
        \supp(\mu_i) \subset \omega_{D_i}(\mu_i) \textrm{ for all } i, 
    \end{equation*}
    where $D_i= \mR^n \setminus \overline{\cup_{l \ne i} \omega(\mu_l)}$.
    Then there is a $m$-phase quadrature domain with respect to $\mu_1,\mu_2,\cdots,\mu_m$.
\end{theorem}
\begin{proof}
By \Cref{lem:average-L-infty} above we may without loss of generality assume that the measures all belong to $L^{\infty}(\mR^n)$.
Let $\bfu=(u_{*,1},u_{*,2},\cdots,u_{*,m})$ be a minimizer of $\mJ_{\bff}$ over $\mS_m$.
For any given $i$ we note that by assumption $\mu_i$ has compact support in $D_i$. 
So in light of \Cref{lem:support-of-measure-minimizer} we have $\mu_i \leq 1$ in $D_i \cap \{u_{*,i}=0\}$.
By (\ref{eq:m-phase-PDE-TS}) it now follows that 
\begin{equation*}
        -\Delta \left(U^{\mu_i}-u_{*,i}\right) \leq 1 \textrm{ in } D_i.
\end{equation*}
Hence $U^{\mu_i}-u_{*,i} \in \mF_{D_i}(\mu)$, so we actually have
\begin{equation*}
   u_{*,i} \ge W_{D_i}^{\mu_i}.
\end{equation*}
So by assumption $\supp{(\mu_i)} \subset \omega_{D_i}(\mu_i) \subset \{u_{*,i}>0\}$. From this the result is an immediate consequence of \Cref{cor:sufficient-mQD} above.
\end{proof}

Note that in the proof above we actually showed that in this situation any minimizer $\bfu$ of $\mJ_{\bff}$ over $\mS_m$ is a $m$-phase quadrature domain. So, in particular, in this situation the minimizer over $\mS_m$ is unique.

The following is now a more or less immediate consequence of \Cref{lem:concentration-condition-one-phase} and the above theorem.
\begin{corollary}\label{cor:mainthm-1}
    Suppose $\mu_1,\mu_2,\ldots,\mu_m$ are such that 
    \begin{equation*}
        \supp(\mu_i) \cap \overline{\omega(\mu_j)}= \emptyset \textrm{ for all } i \ne j 
    \end{equation*}
    and each measure satisfies the concentration condition
    \begin{equation*}
        \limsup_{r \to 0_+} \frac{\mu_i(B_r(x))}{\sfm(B_r)} > 2^n \textrm{ for all } x \in \supp{(\mu_i)}.
    \end{equation*}
    Then there is a $m$-phase quadrature domain with respect to $\mu_1,\mu_2,\cdots,\mu_m$.
\end{corollary}

\Cref{thm:main-1-rough} now follows directly, since $\omega_i$ exists by \Cref{lem:concentration-condition-one-phase} and $\omega_i=\omega(\mu_i)$ necessarily, since this is the only strong one-phase quadrature domain for subharmonic functions with respect to $\mu_i$.

\section{Proof of \texorpdfstring{\Cref{thm:main-2}}{Theorem \ref{thm:main-2}}} 

Throughout this section we let $\bff=(f_{1},\cdots,f_{m})$ be defined as in \eqref{eq:special-choice-f}, i.e., 
\begin{equation*}
f_{j} = \mu_{j} - 1 \quad \text{for all $j=1,\cdots,m$,} 
\end{equation*} 
where, until the proof of \Cref{thm:main-2} at the end of the section, we assume that $\mu_1,\mu_2,\cdots,\mu_m \in L^{\infty}(\mR^n)$ with compact and disjoint supports.
We have already seen the importance of minimizers over the class $\mS_{m,\bmmu}$ in connection with the uniqueness theorem. These will also play a crucial role to get existence results, where we necessarily need additional assumptions on the involved measures $\mu_1,\mu_2,\dots,\mu_m$. We start this section by studying general properties of potential minimizers of $\mJ_{\bff}$ over $\mS_{m,\bmmu}$, and then proceed to state sufficient conditions for the existence of such minimizers and $m$-phase quadrature domains.  

There will always be a minimizing sequence in $\mS_{m,\bmmu}$ which, after passing to a suitable subsequence, converges weakly in $\mS_m$. However, the limit $(u_1,u_2,\cdots,u_m)$ may fail to satisfy $\mu_j(\{u_j=0\})=0$, and therefore may not belong to $\mS_{m,\bmmu}$. So in general we don't have existence of a minimizer over $\mS_{m,\bmmu}$. Furthermore, even if such a minimizer $(u_{*,1},u_{*,2},\cdots,u_{*,m}) \in \mS_{m,\bmmu}$ exists it need not minimize $\mJ_{\bff}$ over the larger set $\mS_{m}$, and in particular not satisfy \eqref{eq:Euler-Lagrange-general} in all of $\mR^n$.

Instead, a minimizer of $\mJ_{\bff}$ over $\mS_{m,\bmmu}$, if it exists, in general satisfies only the following version of the PDE inequality:

\begin{lemma}\label{lem:improvement-prop1-AS16-SMMU}
 If $\bfu_{*}=(u_{*,1},\cdots,u_{*,m})$ is a minimizer of $\mJ_{\bff}$ over $\mS_{m,\bmmu}$, then for each $j=1,\cdots,m$ we have 
\begin{equation}
-\Delta\left(\sum_{\ell\neq j}u_{*,\ell}-u_{*,j}\right) \le \sum_{i\neq j}f_{i}\chi_{\{u_{*,i}>0\}} - f_{j}\chi_{\{u_{*,j}>0\}} \quad \text{in $\left(\mR^{n} \setminus \bigcup_{\ell \ne j} \supp\,(\mu_\ell)\right)$,} \label{eq:Euler-Lagrange-general-smmu}
\end{equation}
and
\[-\Delta u_{*,j}=f_j \textrm{ in } \{u_{*,j}>0\}.\]
\end{lemma}

\begin{remark}
    Note that if we add the inequality and the equality above, then it follows that the only points where the inequality may fail is at points in $\partial \{u_{*,j}>0\} \cap \partial \bigcup_{\ell \ne j} \supp(\mu_{\ell})$, i.e. where $\{u_{*,j}>0\}$ reaches $\supp(\mu_{\ell})$ for some $\ell \ne j$, because the interior of the supports of $\supp(\mu_{\ell})$ is always a subset of $\{u_{*,\ell}>0\}$ basically by definition if $\bfu_* \in \mS_{m,\bmmu}$, at-least if we choose a suitable representative.
\end{remark}

\begin{proof}[Proof of \Cref{lem:improvement-prop1-AS16-SMMU}]
We note that in the proof of \Cref{lem:improvement-prop1-AS16}, if $\bfu_{*} \in \mS_{m,\bmmu}$ and $\psi$ is supported in $\mR^{n} \setminus \bigcup_{\ell \ne j} \supp\,(\mu_\ell)$, then the constructed function $(z_1,z_2,\cdots,z_m)$ also lies in $\mS_{m,\bmmu}$. The inequality in $\mR^{n} \setminus \bigcup_{\ell \ne j} \supp\,(\mu_\ell)$ then follows exactly as in the proof of \Cref{lem:improvement-prop1-AS16}.
Hence we have the result
\[-\Delta\left(\sum_{\ell\neq j}u_{*,\ell}-u_{*,j}\right) \le \sum_{i\neq j}f_{i}\chi_{\{u_{*,i}>0\}} - f_{j}\chi_{\{u_{*,j}>0\}} \quad \text{in $\mR^{n} \setminus \bigcup_{\ell \ne j} \supp\,(\mu_\ell)$.}\]

This allows us to refine the representative of a given minimizer. For, by an essentially identical argument as in the proof of \Cref{cor:cont-min-Sm}, there is a representative of $u_{*,j}$ which is LSC in $\mR^{n} \setminus \bigcup_{\ell \ne j} \supp\,(\mu_\ell)$. But since we have also $u_{*,j}=0$ in $\bigcup_{\ell \ne j} \supp\,(\mu_\ell)$ it follows that it is actually LSC in all of $\mR^n$ with this choice. So in particular the sets $\{u_{*,j}>0\}$ are open subsets of $\mR^{n} \setminus \bigcup_{\ell \ne j} \supp\,(\mu_\ell)$ with this choice of representatives.

Finally, now assume that $\psi \in C_c^{\infty}(\{u_{*,i}>0\})$ (not necessarily positive now), then for all $\epsilon>0$ so small that $u_{*,i}-\epsilon \psi >0$, $(z_1,z_2,\cdots,z_m)$ also belongs to $\mS_{m,\bmmu}$, and again the proof of the inequality
\[\int \nabla u_{*,j} \cdot \nabla \psi \rmd \sfm \ge \int_{\{u_{*,j}>0\}} f_j \psi \rmd \sfm\]
follows exactly as in the proof of \Cref{lem:improvement-prop1-AS16}. Since this may also be applied to $-\psi$ we get  equality above, which proves the final equality in the lemma.
\end{proof}

\begin{corollary} \label{cor:cont-min-Smmu}
     If $\bfu_{*}=(u_{*,1},\cdots,u_{*,m})$ is a minimizer of $\mJ_{\bff}$ over $\mS_{m,\bmmu}$, then it has a representative where each function $u_{*,j}$ is everywhere LSC, identically $0$ on $\bigcup_{\ell \ne j} \supp\,(\mu_\ell)$, and continuous in $\mR^{n} \setminus \left(\bigcup_{\ell \ne j} \supp\,(\mu_\ell) \cap \partial \{u_{*,j}>0\}\right) $. In particular $Q_j=\{u_{*,j}>0\}$ is an open subset of $\mR^n \setminus \bigcup_{\ell \ne j} \supp\,(\mu_{\ell})$, and $-\Delta u_{*,j}=f_j$ in $Q_j$.
\end{corollary}

\begin{proof}
Using the results from the proof of \Cref{lem:improvement-prop1-AS16-SMMU} above, it is clear that we still have $u_{*,j}=U^{f_j|_{Q_j}}- U^{\gamma_j}$ (q.e.) for each $u_{*,j}$, where $\gamma_j$ is supported on $\partial \{u_{*,j}>0\}$. Now the same type of argument as in \Cref{cor:cont-min-Sm} also works to show that if $y \in \partial \{u_{*,j}>0\} \setminus \bigcup_{\ell \ne j} \supp\,(\mu_\ell)$, then $u_{*,j}$ must be continuous in $y$.  (But at the points in $\bigcup_{\ell \ne j} \supp\,(\mu_\ell) \cap \partial \{u_{*,j}>0\}$ this argument fails, and $u_{*,j}$ could potentially be discontinuous at such points if also $\bigcup_{\ell \ne j} \supp\,(\mu_\ell)$ is thin.) 
\end{proof}

If we assume that there is a minimizer $(u_{*,1},u_{*,2},\cdots,u_{*,m})$  over $\mS_{m,\bmmu}$ and let $(v_{*,1},v_{*,2},\cdots,v_{*,m})$ denote the unique minimizer over $\mK^m$, then we would  get with $w_i=u_{*,i}-v_{*,i}$ that
$-\Delta w_i \le 0$ in $\{u_{*,i}>0\}$ (where we used that $\mu_i \le 1$ in $\{v_{*,i}=0\}$). And therefore by the maximum principle it follows that $u_{*,i} \le v_{*,i}$ always holds (just as for the minimizers over $\mS_m$), and in particular the sets $\{u_{*,i}>0\}$ are bounded.

In contrast to minimizers over $\mS_m$ we always have uniqueness of minimizers over $\mS_{m,\bmmu}$.

\begin{lemma}\label{lem:uniq-min-SMMU}
    There is at most one minimizer of $\mJ_{\bff}$ over $\mS_{m,\bmmu}$.
\end{lemma}
\begin{proof}
    This follows almost identically as the proof of \Cref{prop:uniqueness-procedure} above. Indeed,  if we in that proof  assume that $\bfu_*$ is a minimizer of $\mJ_{\bff}$ over $\mS_{m,\bmmu}$, then everything up to equation \eqref{eq:computation1} goes through. Furthermore, using \eqref{eq:Euler-Lagrange-general-smmu} above, also \eqref{eq:computation3} holds by approximation, but now only for all non-negative $\psi \in H^1_0 \left(\left(\mR^{n} \setminus \bigcup_{\ell \ne i} \supp\,(\mu_\ell)\right) \right)$. But since $\{g_i \ne 0\} \subset \{u_{*,i}>0\} \cup\{v_i>0\} \subset \mR^{n} \setminus \bigcup_{\ell \ne i} \supp\,(\mu_\ell)$ we see that $g_i^+$ belongs to this class, so the rest of the argument works the same way.
\end{proof}

The goal for the rest of the section is to give sufficient conditions to ensure that we have a minimizer of $\mJ_{\bff}$ over $\mS_{m,\bmmu}$, and furthermore such that each $\mu_j$ have compact support inside $\{u_j>0\}$.

First of all we need a way to reduce the study of minimizing sequences to special types. To do so we need to extend the definition of $W_D^{\mu}$ to bounded quasiopen $D$, for $\mu \in L^{\infty}(\mR^n)$. To do so note that by definition there is a decreasing sequence of bounded open sets $D^1,D^2,D^3,\cdots$ such that ${\rm Cap}(D^n \setminus D) \to 0$ as $n \to \infty$. We also know that $W_{D^n}^{\mu}$ is decreasing to some function $u$, which then by construction is $0$ q.e. in $\mR^n \setminus D$. It is easy to see that this function minimizes $J_f(v)$, where $f=\mu-1$, over all $v \in H^1_0(D)$, where the latter as usual is defined as all functions $v \in H^1(\mR^n)$ such that $v=0$ q.e. in $\mR^n \setminus D$. It is also clear that this does not depend on the choice of the sequence $D^n$. So we may now define $u=W_D^{\mu}$, and we use this definition in the following statement, where we assume that we have chosen quasicontinuous representatives of the involved functions. 

The following statement is now a more or less immediate consequence of \eqref{eq:minimizerJf} (and the above definition of $W_{D_j}^{\mu_j}$ in case $D_j$ is only quasiopen).

\begin{lemma}\label{lem:struct-of-min-Smmu} 
Let $(v_1,v_2,\cdots,v_m) \in \mS_{m,\bmmu}$ have compact support. If we for each $j=1,\cdots,m$ define $D_j=\{v_j>0\}$ and $u_j=W_{D_j}^{\mu_j}$, then $\mJ_{\bff}(\bfv) \ge \mJ_{\bff}(\bfu)$.
\end{lemma}

Note that we do not claim that $(u_1,u_2,\cdots,u_m)$ in general will belong to $\mS_{m,\bmmu}$ above. But note that if each $D_j$ is open and each measure satisfies the concentration condition in \Cref{lem:concentration-condition-one-phase}\ref{itm:b}, then it will do so automatically. To see this, for each $j=1,\cdots,m$, fix a nonempty compact set $K\subset D_{j}$ and define $\tilde{\mu}_j=\mu_j|_{K}$. 
Then $\tilde{\mu}_j$ satisfies the assumptions of  \Cref{lem:concentration-condition-one-phase}\ref{itm:b}. It follows that $\supp(\tilde{\mu}_j) \subset \omega_{D_j}(\tilde{\mu}_j) \subset \omega_{D_j}(\mu_j)$, and hence $\tilde{\mu}_j(\{u_j=0\})=0$. Since $K$ is arbitrary, the claim follows. 

We now impose the following assumption, which is actually as it turns out a sufficient conditions for the existence of a minimizer over $\mS_{m,\bmmu}$.

\begin{assumption}\label{assu:standing-assumption-point-mass-v2} 
There is a constant $c>1$ and for each $j=1,\cdots,m$, a nonempty bounded open set $A_{j}$ such that 
\begin{equation}\label{eq:conc-cond-Aj}
c\chi_{A_j}\le \mu_j \le \|\mu_j\|_{L^{\infty}(\mR^n)} \chi_{A_j}. 
\end{equation}
\end{assumption}

We then obtain the following result: 

\begin{lemma}\label{lem:ex-of-min-overSmmu}
Suppose \Cref{assu:standing-assumption-point-mass-v2} hold for some $c>1$. Then $\mJ_{\bff}$ admits a minimizer $\bfu_*$ over $\mS_{m,\bmmu}$. 
\end{lemma}
\begin{remark}
It may be worthwhile to note that the sets $D_j^k$ in the below proof may be only quasi-open, but the minimizer $\bfu_*$ will have a LSC representative as stated in \Cref{cor:cont-min-Smmu} above.
\end{remark}

\begin{proof}[Proof of \Cref{lem:ex-of-min-overSmmu}] 
Let $\{(v_1^k,v_2^k,\cdots,v_m^k)\}_{k\in\mN}$ be a minimizing sequence, and let $D_j^k=\{v_j^k>0\}$.
By assumption $A_j \subset D_j^k$ for each $k$. 
By the maximum principle, for each $j=1,\cdots,m$ and $k\in\mN$, we have 
\begin{equation*}
W_{D_j^k}^{\mu_j} \ge w_j=(c-1)G_{A_j}^{\sfm|_{A_j}}, 
\end{equation*}
where $G_{A_j}^{\sfm|_{A_j}}$ is the Green potential of the measure $\sfm|_{A_j}$ over $A_j$, i.e., 
\begin{equation*}
\Delta G_{A_j}^{\sfm|_{A_j}} =-1 \text{ in $A_{j}$}, \textrm{ and } G_{A_j}^{\sfm|_{A_j}}=0 \textrm{ q.e. in } {\mR^{n}\setminus A_{j}}, 
\end{equation*}
which exists since $A_j$ is a nonempty bounded open set (see, e.g., \cite[Chapter~4]{AG01Potential}). 
So if we define $\{(u_1^k,u_2^k,\cdots,u_m^k)\}_{k\in\mN}$ by letting $u_j^k=W_{D_j^k}^{\mu_j}$ as in \Cref{lem:struct-of-min-Smmu}, this sequence is also minimizing, and by the above it belongs to $\mS_{m,\bmmu}$. 
Note that $w_j$ is independent of $k$, and satisfies $w_j>0$ on $A_j$.
Consequently, any convergent subsequence of $u_j^k$ with limit $u_{*,j}$ satisfies $u_{*,j} \ge w_j$, which ensures $\mu_j(\{u_{*,j}=0\})=0$. Therefore, the limit $\bfu_*=(u_{*,1},u_{*,2},\cdots,u_{*,m})$ belongs to $\mS_{m,\bmmu}$. 
\end{proof}

\Cref{assu:standing-assumption-point-mass-v2} above is enough to guarantee existence of a minimizer, but to give conditions to actually get compact support of the measure $\mu_j$ inside $\{u_{*,j}>0\}$ more is needed. Indeed this would be the case already in the one-phase case. The goal below is to show that in case the sets $A_j$ at each boundary point satisfies an inner ball condition, and that the measure $\mu_j$ is sufficiently concentrated then this forces the closure of the ball to be inside the set $\{u_{*,j}>0\}$. In the one-phase situation it is more or less trivial to show that it is enough to assume $\mu_j>c>1$ on this ball. But to quantify what sufficiently concentrated means in the multiphase setting becomes much more involved, since this depends also on both the size and the location of the other measures $\mu_{\ell}$ for $\ell \ne j$.

\begin{theorem} \label{thm:fund-ex-thm1}
Let $\mu_{1},\cdots,\mu_{m}\in L^{\infty}(\mR^{n})$ be non-negative measures with compact and disjoint supports, which satisfies \Cref{assu:standing-assumption-point-mass-v2} for some $c>1$, and let $\bfu_*$ be the unique minimizer of $\mJ_{\bff}$ over $\mS_{m,\bmmu}$.
Furthermore let 
\[R_1>{\rm diam}\left(\bigcup_{i=1}^{m}\supp\,(\mu_{i})\right),\]
and  
\[M = \sup_i \mu_i(\mR^n).\]
Let $j\in\{1,\cdots,m\}$, $x \in \supp(\mu_j)$ and $R>0$ be such that $\mu_{\ell}(B_{2R}(x))=0$ for all $\ell \ne j$. 
If there exists $\delta \in (0,R)$ such that $B_{\delta}(x) \subset A_j$ and $\mu_j \ge C$ in $B_{\delta}(x)$, with 
\begin{equation*} 
C > 1+\frac{M(\Psi(R)-\Psi(2R_2))}{\Psi(\delta)-\Psi(R)} \frac{1}{|B_{\delta}|},
\end{equation*} 
where
\[R_{2}=R_1+\sqrt[n]{\frac{M}{|B_1|}},\] 
then $\overline{B_{\delta}(x)} \subset \{u_{*,j}>0\}$. 
\end{theorem}

\begin{remark}
Note that with this choice of $R_1$ we have    
$\bigcup_{i=1}^{m}\supp\,(\mu_{i}) \subset B_{R_1}(x)$ for any $x \in \bigcup_{i=1}^{m}\supp\,(\mu_{i})$.
Also note that since $c>1$ by assumption it follows from \Cref{lem:ex-of-min-overSmmu} that there is a minimizer $\bfu_*$, and by \Cref{lem:uniq-min-SMMU} it is unique.
We also note that for any  $x \in \bigcup_{i=1}^{m}\supp\,(\mu_{i})$
\begin{equation}
\{W^{\mu_i}>0\} \subset B_{R_{2}}(x) \label{eq:constant-R2}
\end{equation}
for any $x \in \bigcup_{i=1}^{m}\supp\,(\mu_{i})$. This follows from the fact that the mollified measures $\mu_i^r$ satisfy  $\mu_i^r \leq 1$ when $r=\sqrt[n]{M/\abs{B_1}}$, and are supported in a ball of radius $R_1+r$ centered at $x$. Since we have proved above that any minimizer satisfies $u_{*,i} \le W^{\mu_i}$, this gives that the minimizer must be zero outside the ball $B_{R_2}(x)$.
\end{remark}

\begin{proof}[Proof of \Cref{thm:fund-ex-thm1}] 
We may assume without loss of generality that $x=0$, and let $\ell \ne j$ throughout. 
Note that for any $y,z \in B_{R_2}$, we have $|y-z|<2R_2$, so that $\Psi(|y-z|)-\Psi(2R_2) \ge 0$. Hence, defining 
\begin{equation*} 
w_1(y):=\int (\Psi(|y-z|)-\Psi(2R_2))\,\rmd\mu_{\ell}(z), 
\end{equation*} 
we have $w_{1}(y)\ge 0$ for all $y\in\overline{B_{R_2}}$. 
Since $-\Delta w_1=\mu_{\ell}$ and $D_{\ell}\subset B_{R_2}$, we have 
\begin{equation*}
-\Delta(w_{1}-u_{*,\ell}) = \mu_{\ell} + \Delta W_{D_{\ell}}^{\mu_{\ell}} = \Bal_{D_{\ell}}(\mu_{\ell})\ge 0  \text{ in $D_{\ell}$} ,\quad (w_{1}-u_{*,\ell})|_{\partial D_{\ell}}\ge 0. 
\end{equation*}
It then follows from the maximum principle that $u_{*,\ell} \le w_1$ in $D_{\ell}$, and hence in all of $B_{R_2}$. 
Since $R<R_2$, for each $y\in\partial B_R$ we have 
\begin{equation*} 
\begin{aligned} 
& u_{*,\ell}(y) \le w_1(y) = \int_{\mR^{n}\setminus B_{2R}}(\Psi(|y-z|)-\Psi(2R_2))\,\rmd\mu_{\ell}(z) \quad \text{\footnotesize (since $\mu_{\ell}(B_{2R})=0$)} \\ 
&\le(\Psi(R)-\Psi(2R_2))  \int_{\mR^{n}\setminus B_{2R}}\,\rmd\mu_{\ell}(z) \quad \text{\footnotesize (since $\abs{z-y}\ge R$ for all $z\notin B_{2R}$)} \\ 
&\le M(\Psi(R)-\Psi(2R_2)) \quad \text{\footnotesize (since $\norm{\mu_{i}}_{L^{1}(\mR^{n})}\le M$)}. 
\end{aligned} 
\end{equation*} 
Next, define the function 
\begin{equation*} 
w_2(y):=\frac{M(\Psi(R)-\Psi(2R_2)}{\Psi(\delta)-\Psi(R)}(\Psi(\delta)-\Psi(|y|))
\end{equation*} 
in order to satisfy $w_2|_{\partial B_R}=M(\Psi(R)-\Psi(2R_2)$, then
\begin{equation*} 
w_2(y) \ge w_1(y) \ge u_{*,\ell}(y) \quad \text{for all $y\in\partial B_R$}.
\end{equation*} 
Since $B_{\delta}\subset A_{j}\subset D_{j}$, we have $u_{*,\ell}|_{B_{\delta}}=0$, therefore $D_{\ell}\cap B_{\delta}=\emptyset$. 
Since $\mu_{\ell}(B_{2R})=0$, we see that 
\begin{equation*}
-\Delta (w_{2}-u_{*,\ell}) = \Delta W_{D_{\ell}}^{\mu_{\ell}} = \Bal_{D_{\ell}}(\mu_{\ell}) \ge 0 \text{ in $B_{R}\cap D_{\ell}$} ,\quad (w_{2}-u_{*,\ell})|_{\partial(B_{R}\cap D_{\ell})} \ge 0. 
\end{equation*}
because $w_{2}>0$ in $B_{R}\setminus\overline{B_{\delta}}$. 
Another application of the maximum principle gives that $u_{*,\ell}\le w_{2}$ in $B_{R}\cap D_{\ell}$, and hence in all of $B_{R}\setminus\overline{B_{\delta}}$. 

Since $\mu_{j}\ge C > 1$ in $B_{\delta}\subset D_{j}$, we have 
\begin{equation*}
-\Delta u_{*,j} = -\Delta W_{D_{j}}^{\mu_{j}} = \mu_{j}-\overbrace{\Bal_{D_{j}}(\mu_{j})}^{=\,1} \ge C-1 \quad \text{in $B_{\delta}$.} 
\end{equation*}
Hence, 
\begin{equation*} 
-\Delta \left( u_{*,j} - \frac{C-1}{2n}(\delta^2-|\cdot|^2) \right) \ge 0 \text{ in $B_{\delta}$} ,\quad \left.\left( u_{*,j} - \frac{C-1}{2n}(\delta^2-|\cdot|^2) \right)\right|_{\partial B_{\delta}}\ge 0.  
\end{equation*} 
By the maximum principle, it follows that 
\begin{equation*} 
\frac{C-1}{2n}(\delta^2-|y|^2) \le u_{*,j}(y) \quad \text{for all $y\in B_{\delta}$.}
\end{equation*}

Suppose, for the sake of contradiction, that there exists $y_0 \in \partial B_{\delta} \cap \{u_{*,j}=0\}$. We now introduce the functions 
\begin{equation*} 
s(y):=\sum_{\ell \ne j} u_{*,\ell}(y) - u_{*,j}(y) + \frac{|y-y_0|^2}{2n},
\end{equation*} 
and
\begin{equation*} 
 w_3(y) := \left\{\begin{aligned}
& w_2(y) + \frac{|y-y_0|^2}{2n} && y \in B_R \setminus B_{\delta},\\
& -\frac{C-1}{2n}(\delta^2-|y|^2)+\frac{|y-y_0|^2}{2n} && y \in B_{\delta}.
\end{aligned}\right. 
\end{equation*}
First of all a direct computation (using \Cref{lem:improvement-prop1-AS16-SMMU} and that $f_{\ell}=-1$ in $B_R$) shows that $s$ is subharmonic in $B_R$. 
Furthermore it follows from the above that $s \le w_3$ in $B_R$. This means that for small $\epsilon$ we will have
\[\int_{B_{\epsilon}(y_0)} s \,\rmd\sfm \le \int_{B_{\epsilon}(y_0)} w_3 \,\rmd\sfm.\]
So if we can prove that the latter for small $\epsilon$ is negative, then we have a contradiction to the mean value inequality since $s(y_0)=0$.
To compute the jump in the normal (i.e. radial) derivative of $w_3$ in $y_0$ we first note that
\begin{equation*} 
\frac{\partial (w_2 + |\cdot-y_0|)}{\partial r}(y_0) = \frac{M(\Psi(R)-\Psi(2R_1))}{\Psi(\delta)-\Psi(R)} \frac{1}{n|B_1|\delta^{n-1}},
\end{equation*} 
where we used that $-\psi'(\delta)=1/(n|B_1|\delta^{n-1})$. 
Then we also note that by assumption
\begin{equation*} 
\frac{(C-1)\delta}{n} > \frac{M(\Psi(R)-\Psi(2R_2))}{\Psi(\delta)-\Psi(R)}\frac{1}{|B_{\delta}|}\frac{\delta}{n} = \frac{M(\Psi(R)-\Psi(2R_2))}{\Psi(\delta)-\Psi(R)}\frac{1}{n|B_1|\delta^{n-1}},
\end{equation*} 
so
\begin{equation*} 
\frac{\partial (-(C-1)(R^2-(\cdot)^2)/2n)}{\partial r}(y_0)=\frac{(C-1)\delta}{n} > \frac{M(\Psi(R)-\Psi(2R_1))}{\Psi(\delta)-\Psi(R)}\cdot \frac{1}{n|B_1|\delta^{n-1}}.
\end{equation*}
Hence the assumption on $C$ forces a jump in the normal derivative, and from this it easily follows that we for $\epsilon$ small enough must have 
\begin{equation*} 
\int_{B_{\epsilon}(y_0)} w_3 \,\rmd\sfm<0
\end{equation*}
and the proof is done.
\end{proof}

A more or less immediate consequence of the above is the following.
\begin{corollary}\label{cor:main-exist-cor}
Suppose \Cref{assu:standing-assumption-point-mass-v2} hold for some $c>1$, and let $\bfu_*$ be the unique minimizer of $\mJ_{\bff}$ over $\mS_{m,\bmmu}$. Assume furthermore that there for each $j$ and every point $y_0 \in \partial A_j$ is a ball $B_{\delta}(x) \subset A_j$  with  $y_0 \in \partial B_{\delta}(x)$, such that the assumptions in \Cref{thm:fund-ex-thm1} are satisfied.
     Then $\supp(\mu_j)\subset \{u_{*,j}>0\}$, and in particular there is a $m$-phase quadrature domain with respect to $\mu_1,\mu_2,\cdots,\mu_m$.
\end{corollary}

\begin{proof}[Proof of \Cref{thm:main-2}] 
By assumption the measures are of the form 
\[\mu_j=\sum_{k=1}^{m_j} c^j_k \delta_{x^j_k},\]
where $c^j_k>0$ for each $j,k$. It is easy to see that in case $\delta>0$ is small enough, then $\mu_1^{\delta},\mu_2^{\delta},\cdots,\mu_m^{\delta}$ satisfies the assumptions in \Cref{cor:main-exist-cor}. Indeed as long as $\delta$ is smaller than the distance between any of the points $x^j_k$ we have 
\[\mu_j^{\delta}= \sum_{k} c^j_k|B_{\delta}|^{-1}\sfm|_{B_{\delta}(x^j_k)},\]
and these measures have disjoint supports. Furthermore
\[\frac{c_k^j}{|B_{\delta}|} >1 +\frac{M(\psi(R)-\psi(2R_2))}{\psi{\delta}-\psi(R)}\frac{1}{|B_{\delta}|} \iff c_k^j > |B_{\delta}| + \frac{M(\psi(R)-\psi(2R_2))}{\psi(\delta)-\psi(R)}.\]
But the right hand side above goes towards $0$ as $\delta \to 0$, and therefore we get the needed inequality for $\delta$ small enough.

But if we for such $\delta$ let $\bfu_{*}^{\delta}$ denote the corresponding minimizer and simply define $u_{*,j}=u_{*,j}^{\delta} + U^{\mu_j} - U^{\mu_j^{\delta}}$, then $u_{*,j} \ge u_{*,j}^{\delta}$ with equality outside $\bigcup_k B_{\delta}(x^j_k)$, and it follows immediately that this defines a $m$-phase quadrature domain with respect to $\mu_1,\mu_2,\cdots,\mu_m$.
\end{proof}

\section{Junction point in a symmetric multiphase quadrature domain} 

The aim of this section is to examine some highly symmetric cases in the plane. For convenience, we work with the complex variable $z=x+\bfi y=re^{\bfi\theta}$ throughout this proof (so, for instance, the point $(1,0)$ is simply written as $1$ below). 
More precisely, we consider the case where $\mu_1,\mu_2,\ldots,\mu_m$, with $m\ge 2$, are point masses of equal size, evenly distributed on the unit circle. That is,
\begin{equation*}
\mu_j=C\delta_{e^{\bfi(j-1)\theta_0}} ,\quad \theta_0 =2\pi/m
\end{equation*}
where $C>0$ is a constant. We already know that an $m$-phase quadrature domain exists in this situation. By symmetry, it follows that 
\begin{equation*}
u_j(re^{\bfi\theta})=u_1(re^{\bfi(\theta-(j-1)\theta_0)}). 
\end{equation*}
In particular, the sets  $\{u_j>0\}$ are rotations of each other. Since $\{u_1>0\}$ is clearly symmetric with respect to the $x$-axis, it follows that 
\begin{equation*}
Q_1:=\{u_1>0\} \subset D_1:= \left\{re^{\bfi \theta}: 0<r<\infty, \abs{\theta} < \frac{\pi}{m} \right\},
\end{equation*}
and the remaining domains $Q_2,\cdots,Q_m,D_2,\cdots,D_m$ are obtained by rotating this set successively by the angle $2\pi/m$. 
Using \eqref{eq:fixed-point}, one sees that 
\begin{equation*}
(Q_1,\cdots,Q_m)=(\omega_{D_1}(\mu_1),\cdots,\omega_{D_m}(\mu_m)) 
\end{equation*}
forms a strong $m$-phase quadrature domain with respect to $(\mu_1,\cdots,\mu_m)$, with $u_1=W_D^{\mu_1}$. In other words, this case can in fact be handled using one-phase partial balayage.

Unlike the case in \Cref{ex:counterexample-tpqd-v3}, when $m>2$ this example does not in general arise simply by splitting a two-phase quadrature domain into several pieces. This naturally leads to the following question: 
\begin{equation*}
\text{Does the origin (i.e. junction point) belong to $\partial Q_1$ when $C$ is sufficiently large?
}
\end{equation*}
We now show that $B_{\varepsilon}(0) \cap D \subset \omega_D(\mu_1)$ for large enough $C>0$ when $m=3$, but that this never occurs for $m \geq 4$. In view of the above discussion, this follows as a special case of the following proposition.

\begin{proposition}
Let $D=\{re^{\bfi\theta}: r >0, |\theta| < \theta_0\}$, where $\theta_0 \in (0,\pi/2)$. Then there exists $C>1$ such that $0 \in \overline{\omega_D(C\delta_{1})}$ if and only if $\theta_0 >\pi/4$. 
\end{proposition}

\begin{proof}
First, assume that $\theta_{0}>\pi/4$. By the definition of partial balayage, we have 
\begin{equation*}
-\Delta W^{\pi(\delta_{1}+\delta_{-1})} = \delta_{1}+\delta_{-1} - \sfm|_{B_{1}(1)\cup B_{1}(-1)}, \quad W^{\pi(\delta_{1}+\delta_{-1})}|_{\mR^{n}\setminus(B_{1}(1)\cup B_{1}(-1))} = 0. 
\end{equation*}
Recall that $\frac{1}{2\pi}\log\abs{z}$ is a fundamental solution of $-\Delta$. For each $C>1$, define 
\begin{equation*}
F_{C}(z) := W^{\pi\cdot(\delta_{1}+\delta_{-1})}(z) - \frac{C-1}{2\pi}\ln\abs{z^{2}-1}. 
\end{equation*}
Then $F_{C}$ satisfies 
\begin{equation*}
-\Delta F_{C} = C\delta_{1}+C\delta_{-1} - \sfm|_{B_{1}(1)\cup B_{1}(-1)}  
\end{equation*}
and 
\begin{equation*}
F_{C}(z) = - \frac{C-1}{2\pi}\ln\abs{z^{2}-1} \quad \text{for all $z\in \mR^{n}\setminus(B_{1}(1)\cup B_{1}(-1))$.} 
\end{equation*}
Moreover, 

\begin{equation*}
-\frac{C-1}{2\pi}\ln\abs{z^2-1} = \frac{C-1}{2\pi}(x^{2}-y^{2}) + \mathcal{O}(\abs{z}^{4}) \quad \text{for all $z$ near the origin},
\end{equation*}
and 
\begin{equation*}
W^{\pi(\delta_{1}+\delta_{-1})}(z) = x^{2} + \mathcal{O}(\abs{z}^{3})\quad \text{for all $z$ near the origin}.
\end{equation*}
From these expansions it follows that, if $C$ is chosen sufficiently large, 
\begin{equation*}
\{F_C>0\}\cap\{x>0\} = \left\{ z=x+\bfi y : x>0, F_C(z)>0 \right\} \subset D 
\end{equation*}
since $\theta_{0}>\pi/4$. Furthermore, 
\begin{equation*}
W_{D}^{C\delta_{1}} \ge 0=F_C \quad \text{on $\partial(\{F_C>0\}\cap\{x>0\})$}
\end{equation*}
and 
\begin{equation*}
-\Delta W_{D}^{C\delta_{1}} \ge -\Delta F_{C} \quad \text{in $\{F_C>0\}\cap\{x>0\}$}. 
\end{equation*}
By the maximum principle, it follows that 
\begin{equation*}
W_{D}^{C\delta_{1}} \ge F_{C} > 0 \quad \text{in $\{F_C>0\}\cap\{x>0\}$.} 
\end{equation*}
In other words, $\{F_C>0\}\cap\{x>0\}\subset\omega_{D}(C\delta_{1})$, which implies that $0\in\overline{\omega_{D}(C\delta_{1})}$. 

Next, we show that if $\theta_{0}\le\pi/4$, then $0\notin\overline{\omega_{D}(D\delta_{1})}$ for any $C>0$. By \eqref{eq:monotonicity-partial-reduction}, it suffices to consider $\theta_{0}=\pi/4$. Let $s\in(0,1)$ be a parameter to be chosen later, and define  
\begin{equation*}
f_{s}(re^{\bfi\theta}) = \left\{\begin{aligned}
& h_{s}(r)\cos(2\theta), && s\le r\le 1, \\ 
& h_{s}(1)\cos(2\theta), && r>1, \\ 
& 0, && 0<r<s, 
\end{aligned}\right. \quad \abs{\theta}<\pi/4, 
\end{equation*}
with 
\begin{equation*}
h_{s}(r) = \frac{1}{4}r^{2}\ln\left(\frac{r}{s}\right) - \frac{r^{2}}{16} + \frac{s^{4}}{16r^{2}}. 
\end{equation*}
Since $\theta_{0}=\pi/4$, we have $f_{s}=0$ on $\partial D$. A direct computation gives 
\begin{equation*}
\Delta f_{s}(re^{\bfi\theta}) = \left(h_{s}''(r) + \frac{h_{s}'(r)}{r} - \frac{4h_{s}(r)}{r^{2}} \right)\cos(2\theta) = \cos(2\theta) 
\end{equation*}
for all $r\in(s,1)$ and $\abs{\theta}<\pi/4$. On the other hand, we observe that $f_{s}|_{\partial B_{s}}=0$ and $\nabla f_{s}|_{\partial B_{s}}=0$, which shows that $\Delta f_{s}$ does not place any charge on $\partial B_{s}\cap D$. 
Moreover, 
\begin{equation*}
\partial_{r} f_{s}(re^{\bfi\theta})|_{r=1} = h_{s}'(1)\cos(2\theta). 
\end{equation*}
Thus, in total,  
\begin{equation*}
-\Delta f_{s}(re^{\bfi\theta}) = \gamma_{s} - \cos(2\theta)\sfm|_{B_{1}\setminus\overline{B_{s}}} \quad \text{in $D$,} 
\end{equation*}
where 
\begin{equation*}
\gamma_{s} = h_{s}'(1)\cos(2\theta)\mathcal{H}^{1}|_{\partial B_{1}}. 
\end{equation*}

Since $W_{D}^{\gamma_{s}}$ is the minimal non-negative function satisfying 
\begin{equation*}
-\Delta W^{\gamma_{s}} \ge \gamma_{s}-1 \quad \text{in $D$}, 
\end{equation*}
we have $W_{D}^{\gamma_{s}} \le f_{s}$ in $D$, and hence $\omega_{D}(\gamma_{s}) \subset \mR^{n}\setminus\overline{B_{s}}$. 
On the other hand, 
\begin{equation*}
h_{s}'(1) = \frac{1}{2}\ln\left(\frac{1}{s}\right) + \frac{1}{8} - \frac{s^{4}}{8} \rightarrow +\infty \quad \text{as $s\rightarrow 0_{+}$,} 
\end{equation*}
so for any given $C>0$, we can choose $s>0$ sufficiently small so that 
\begin{equation*}
\omega_{D}(C\delta_{1})\subset\omega_{D}(\gamma_{s}) \subset \mR^{n}\setminus\overline{B_{s}}. 
\end{equation*}
This shows that $0\notin\overline{\omega_{D}(D\delta_{1})}$ for all $C>0$, completing the proof. 
\end{proof}

\subsection*{Acknowledgments} 

P.-Z. Kow was supported by the National Science and
Technology Council of Taiwan (NSTC 112-2115-M-004-004-MY3), and by the National Center for Theoretical Sciences of Taiwan. H. Shahgholian was supported by Swedish Research Council (grant no. 2021-03700). 

\subsection*{Data Availability}
No data were used or created in the course of this research.

\subsection*{Competing Interests}
The authors declare no conflict of interest.

\end{sloppypar}

\bibliographystyle{custom}
\bibliography{ref}

\begin{thebibliography}{KLSS24}

\bibitem[AS16]{AS16MultiPhaseQD}
A.~Arakelyan and H.~Shahgholian.
\newblock Multi-phase quadrature domains and a related minimization problem.
\newblock {\em Potential Anal.}, 45(1):135--155, 2016.
\newblock
  \href{https://mathscinet.ams.org/mathscinet/article?mr=3511808}{MR3511808},
  \href{https://zbmath.org/1346.35237}{Zbl:1346.35237},
  \href{https://doi.org/10.1007/s11118-016-9539-0}{doi:10.1007/s11118-016-9539-0},
  \href{https://arxiv.org/abs/1511.02779}{\texttt{arXiv:1511.02779}}.

\bibitem[AG01]{AG01Potential}
D.~H. Armitage and S.~J. Gardiner.
\newblock {\em Classical potential theory}.
\newblock Springer Monogr. Math. Springer-Verlag London, Ltd., London, 2001.
\newblock
  \href{https://mathscinet.ams.org/mathscinet/article?mr=1801253}{MR1801253},
  \href{https://zbmath.org/0972.31001}{Zbl:0972.31001},
  \href{https://doi.org/10.1007/978-1-4471-0233-5}{doi:10.1007/978-1-4471-0233-5}.

\bibitem[BP04]{BP04KatoInequality}
H.~Brezis and A.~Ponce.
\newblock Kato's inequality when $\delta u$ is a measure.
\newblock {\em C. R. Math. Acad. Sci. Paris}, 338(8):599--604, 2004.
\newblock
  \href{https://mathscinet.ams.org/mathscinet-getitem?mr=2056467}{MR2056467},
  \href{https://zbmath.org/1101.35028}{Zbl:1101.35028},
  \href{https://doi.org/10.1016/j.crma.2003.12.032}{doi:10.1016/j.crma.2003.12.032},
  \href{https://arxiv.org/abs/1312.6498}{\texttt{arXiv:1312.6498}}.

\bibitem[CTV05]{CTV05SegregationProblem}
M.~Conti, S.~Terracini, and G.~Verzini.
\newblock A variational problem for the spatial segregation of
  reaction-diffusion systems.
\newblock {\em Indiana Univ. Math. J.}, 54(3):779--815, 2005.
\newblock
  \href{https://mathscinet.ams.org/mathscinet/article?mr=2151234}{MR2151234},
  \href{https://zbmath.org/1132.35397}{Zbl:1132.35397},
  \href{https://doi.org/10.1512/iumj.2005.54.2506}{doi:10.1512/iumj.2005.54.2506},
  \href{https://arxiv.org/abs/math/0312210}{\texttt{arXiv:math/0312210}}.

\bibitem[EPS11]{EPS11TwoPhaseQD}
B.~Emamizadeh, J.~V. Prajapat, and H.~Shahgholian.
\newblock A two phase free boundary problem related to quadrature domains.
\newblock {\em Potential Anal.}, 34(2):119--138, 2011.
\newblock
  \href{https://mathscinet.ams.org/mathscinet/article?mr=2754967}{MR2754967},
  \href{https://zbmath.org/1216.35161}{Zbl:1216.35161},
  \href{https://doi.org/10.1007/s11118-010-9184-y}{doi:10.1007/s11118-010-9184-y}.

\bibitem[GS09]{GS09PartialBalayage}
S.~Gardiner and T.~Sj{\"{o}}din.
\newblock Partial balayage and the exterior inverse problem of potential
  theory.
\newblock In {\em Potential theory and stochastics in Albac}, volume~11, pages
  111--123, Theta, Bucharest, 2009. Theta Ser. Adv. Math.
\newblock
  \href{https://mathscinet.ams.org/mathscinet-getitem?mr=2681841}{MR2681841},
  \href{https://zbmath.org/1199.31009}{Zbl:1199.31009}.

\bibitem[GS12]{GS12TwoPhaseQD}
S.~J. Gardiner and T.~Sj{\"o}din.
\newblock Two-phase quadrature domains.
\newblock {\em J. Anal. Math.}, 116:335--354, 2012.
\newblock
  \href{https://mathscinet.ams.org/mathscinet/article?mr=2892623}{MR2892623},
  \href{https://zbmath.org/1288.31002}{Zbl:1288.31002},
  \href{https://doi.org/10.1007/s11854-012-0009-3}{doi:10.1007/s11854-012-0009-3}.

\bibitem[GS14]{GS14Stationary}
S.~J. Gardiner and T.~Sj{\"o}din.
\newblock Stationary boundary points for a {L}aplacian growth problem in higher
  dimensions.
\newblock {\em Arch. Ration. Mech. Anal.}, 213(2):503--526, 2014.
\newblock
  \href{https://mathscinet.ams.org/mathscinet/article?mr=3211858}{MR3211858},
  \href{https://zbmath.org/1308.35210}{Zbl:1308.35210},
  \href{https://doi.org/10.1007/s00205-014-0750-0}{doi:10.1007/s00205-014-0750-0}.

\bibitem[GS25]{GS24PartialBalayageHelmholtz}
S.~J. Gardiner and T.~Sj{\"{o}}din.
\newblock Partial balayage for the {H}elmholtz equation.
\newblock {\em Potential Anal.}, 63(4):1671--1697, 2025.
\newblock
  \href{https://mathscinet.ams.org/mathscinet/article?mr=4990500}{MR4990500},
  \href{https://zbmath.org/8127979}{Zbl:8127979},
  \href{https://doi.org/10.1007/s11118-025-10217-0}{doi:10.1007/s11118-025-10217-0},
  \href{https://arxiv.org/abs/2404.05552}{\texttt{arXiv:2404.05552}}.

\bibitem[GT01]{GT01Elliptic}
D.~Gilbarg and N.~S. Trudinger.
\newblock {\em Elliptic partial differential equations of second order (reprint
  of the 1998 edition)}, volume 224 of {\em Classics in Mathematics}.
\newblock Springer-Verlag Berlin Heidelberg, 2001.
\newblock
  \href{https://mathscinet.ams.org/mathscinet-getitem?mr=1814364}{MR1814364},
  \href{https://zbmath.org/1042.35002}{Zbl:1042.35002},
  \href{https://doi.org/10.1007/978-3-642-61798-0}{doi:10.1007/978-3-642-61798-0}.

\bibitem[Gus90]{Gus90QuadratureDomains}
B.~Gustafsson.
\newblock On quadrature domains and an inverse problem in potential theory.
\newblock {\em J. Analyse Math.}, 55:172--216, 1990.
\newblock
  \href{https://mathscinet.ams.org/mathscinet-getitem?mr=1094715}{MR1094715},
  \href{https://zbmath.org/0745.31002}{Zbl:0745.31002},
  \href{https://doi.org/10.1007/BF02789201}{doi:10.1007/BF02789201}.

\bibitem[Gus04]{Gus04LecturesBalayage}
B.~Gustafsson.
\newblock Lectures on balayage.
\newblock In {\em Clifford algebras and potential theory}, volume~7 of {\em
  Univ. Joensuu Dept. Math. Rep. Ser.}, pages 17--63. Univ. Joensuu, Joensuu,
  2004.
\newblock
  \href{https://mathscinet.ams.org/mathscinet-getitem?mr=2103705}{MR2103705},
  \href{https://zbmath.org/1088.31001}{Zbl:1088.31001},
  \href{http://kth.diva-portal.org/smash/record.jsf?pid=diva2%3A492834}{diva2:492834}.

\bibitem[GS94]{GS94PartialBayage}
B.~Gustafsson and M.~Sakai.
\newblock Properties of some balayage operators, with applications to
  quadrature domains and moving boundary problems.
\newblock {\em Nonlinear Anal.}, 22(10):1221--1245, 1994.
\newblock
  \href{https://mathscinet.ams.org/mathscinet/article?mr=1279981}{MR1279981},
  \href{https://zbmath.org/0852.35144}{Zbl:0852.35144},
  \href{https://doi.org/10.1016/0362-546X(94)90107-4}{doi:10.1016/0362-546X(94)90107-4}.

\bibitem[GS05]{GS05QuadratureDomain}
B.~Gustafsson and H.~S. Shapiro.
\newblock What is a quadrature domain?
\newblock In {\em Quadrature domains and their applications}, volume 156 of
  {\em Oper. Theory Adv. Appl.}, pages 1--25. Birkh\"{a}user, Basel, 2005.
\newblock
  \href{https://mathscinet.ams.org/mathscinet-getitem?mr=2129734}{MR2129734},
  \href{https://zbmath.org/1086.30002}{Zbl:1086.30002},
  \href{https://doi.org/10.1007/3-7643-7316-4_1}{doi:10.1007/3-7643-7316-4\_1}.

\bibitem[H{\"o}r03]{Hoermander_book_1}
L.~H{\"o}rmander.
\newblock {\em The analysis of linear partial differential operators {I}.
  {D}istribution theory and {F}ourier analysis}.
\newblock Classics Math. Springer-Verlag, Berlin, second edition, 2003.
\newblock
  \href{https://mathscinet.ams.org/mathscinet/article?mr=1996773}{MR1996773},
  \href{https://zbmath.org/1028.35001}{Zbl:1028.35001},
  \href{https://doi.org/10.1007/978-3-642-61497-2}{doi:10.1007/978-3-642-61497-2}.

\bibitem[KLSS24]{KLSS22QuadratureDomain}
P.-Z. Kow, S.~Larson, M.~Salo, and H.~Shahgholian.
\newblock Quadrature domains for the {H}elmholtz equation with applications to
  non-scattering phenomena.
\newblock {\em Potential Anal.}, 60(1):387--424, 2024.
\newblock
  \href{https://mathscinet.ams.org/mathscinet/article?mr=4696043}{MR4696043},
  \href{https://zbmath.org/7798456}{Zbl:7798456}
  \href{https://doi.org/10.1007/s11118-022-10054-5}{doi:10.1007/s11118-022-10054-5}.
  The results in the appendix are well-known, and the proofs can found at
  \href{https://arxiv.org/abs/2204.13934}{\texttt{arXiv:2204.13934}}.

\bibitem[KSS24]{KSS23Minimization}
P.-Z. Kow, M.~Salo, and H.~Shahgholian.
\newblock A minimization problem with free boundary and its application to
  inverse scattering problems.
\newblock {\em Interfaces Free Bound.}, 26(3):415--471, 2024.
\newblock
  \href{https://mathscinet.ams.org/mathscinet/article?mr=4762088}{MR4762088},
  \href{https://zbmath.org/7902359}{Zbl:7902359},
  \href{https://doi.org/10.4171/ifb/515}{doi:10.4171/ifb/515},
  \href{https://arxiv.org/abs/2303.12605}{\texttt{arXiv:2303.12605}}.

\bibitem[KS24]{KS24MultiplaseQD}
P.-Z. Kow and H.~Shahgholian.
\newblock Multi-phase $k$-quadrature domains and applications to acoustic waves
  and magnetic fields.
\newblock {\em Partial Differ. Equ. Appl.}, 5(3), 2024.
\newblock Paper No. 13,
  \href{https://mathscinet.ams.org/mathscinet/article?mr=4732406}{MR4732406},
  \href{https://zbmath.org/1540.35127}{Zbl:1540.35127},
  \href{https://doi.org/10.1007/s42985-024-00283-1}{doi:10.1007/s42985-024-00283-1},
  \href{https://arxiv.org/abs/2401.13279}{\texttt{arXiv:2401.13279}}.

\bibitem[PSU12]{PSU12FreeBoundary}
A.~Petrosyan, H.~Shahgholian, and N.~Uraltseva.
\newblock {\em Regularity of free boundaries in obstacle-type problems}, volume
  136 of {\em Graduate Studies in Mathematics}.
\newblock American Mathematical Society, Providence, RI, 2012.
\newblock
  \href{https://mathscinet.ams.org/mathscinet-getitem?mr=2962060}{MR2962060},
  \href{https://zbmath.org/1254.35001}{Zbl:1254.35001},
  \href{http://dx.doi.org/10.1090/gsm/136}{doi:10.1090/gsm/136}.

\bibitem[Sak10]{Sa10SmallMod}
M.~Sakai.
\newblock Small modifications of quadrature domains.
\newblock {\em Mem. Amer. Math. Soc.}, 206(969), 2010.
\newblock
  \href{https://mathscinet.ams.org/mathscinet/article?mr=2667421}{MR2667421},
  \href{https://zbmath.org/1198.31001}{Zbl:1198.31001},
  \href{https://doi.org/10.1090/S0065-9266-10-00596-X}{doi:10.1090/S0065-9266-10-00596-X}.

\end{thebibliography}
\end{document}